\documentclass[a4paper,11pt]{amsart}
\usepackage{amsmath, amsfonts, amssymb, amsthm, amscd}
\usepackage{graphicx}
\usepackage{enumerate}
\usepackage{url}
\usepackage{color}
\usepackage[utf8]{inputenc}
\usepackage{tikz}
\usetikzlibrary{patterns}
\usetikzlibrary{arrows} 
\usepackage[a4paper,scale={0.72,0.74},marginratio={1:1},footskip=7mm,headsep=10mm]{geometry}

\numberwithin{equation}{section}

\newtheorem{theorem}{Theorem}[section]
\newtheorem{lemma}[theorem]{Lemma}
\newtheorem{proposition}[theorem]{Proposition}

\newtheorem{remark}[theorem]{Remark}

\newtheorem{conjecture}[theorem]{Conjecture}

\newcommand{\bbE}{{\ensuremath{\mathbb E}} }

\newcommand{\bbK}{{\ensuremath{\mathbb K}} }

\newcommand{\bbN}{{\ensuremath{\mathbb N}} }

\newcommand{\bbP}{{\ensuremath{\mathbb P}} }

\newcommand{\bbR}{{\ensuremath{\mathbb R}} }

\newcommand{\bbZ}{{\ensuremath{\mathbb Z}} }

\newcommand{\cD}{{\ensuremath{\mathcal D}} }

\newcommand{\cF}{{\ensuremath{\mathcal F}} }
\newcommand{\cG}{{\ensuremath{\mathcal G}} }

\newcommand{\cI}{{\ensuremath{\mathcal I}} }

\newcommand{\cL}{{\ensuremath{\mathcal L}} }

\newcommand{\cP}{{\ensuremath{\mathcal P}} }

\newcommand{\cZ}{{\ensuremath{\mathcal Z}} }

\newcommand{\ga}{\alpha}
\newcommand{\gb}{\beta}
\newcommand{\gga}{\gamma}

\newcommand{\gd}{\delta}
\newcommand{\gD}{\Delta}
\newcommand{\gep}{\varepsilon} 
\newcommand{\gz}{\zeta}
\newcommand{\gt}{\theta}

\newcommand{\gl}{\lambda}

\newcommand{\gs}{\sigma}

\newcommand{\go}{\omega}

\renewcommand{\tilde}{\widetilde}          
\DeclareMathSymbol{\leqslant}{\mathalpha}{AMSa}{"36} 
\DeclareMathSymbol{\geqslant}{\mathalpha}{AMSa}{"3E} 
\DeclareMathSymbol{\eset}{\mathalpha}{AMSb}{"3F}     
\newcommand{\dd}{\text{\rm d}}             

\newcommand{\R}{\mathbb{R}}

\newcommand{\N}{\mathbb{N}}

\newcommand{\PEfont}{\mathrm}
\DeclareMathOperator{\var}{\ensuremath{\PEfont Var}}
\DeclareMathOperator{\cov}{\ensuremath{\PEfont Cov}}
\newcommand{\p}{\ensuremath{\PEfont P}}

\DeclareMathOperator{\sign}{sign}
\DeclareMathOperator{\re}{Re}
\DeclareMathOperator{\im}{Im}

\newcommand\bP{\ensuremath{\mathrm{P}}}
\newcommand\bE{\ensuremath{\mathrm{E}}}

\newcommand{\ind}{{\sf 1}}

\renewcommand{\epsilon}{\varepsilon}

\newcommand{\ef}{\mathrm{eff}}
\newcommand{\hu}{\hat u}
\newcommand{\hU}{\hat U}

\newenvironment{myenumerate}{
\renewcommand{\theenumi}{\arabic{enumi}}
\renewcommand{\labelenumi}{{\rm(\theenumi)}}
\begin{list}{\labelenumi}
{
\setlength{\itemsep}{0.4em}
\setlength{\topsep}{0.5em}
\setlength\leftmargin{2.45em}
\setlength\labelwidth{2.05em}
\setlength{\labelsep}{0.4em}
\usecounter{enumi}
}
}
{\end{list}
}

\renewenvironment{enumerate}{
\begin{myenumerate}}
{\end{myenumerate}}

\newcommand{\htau}{\hat\tau}

\newcommand{\hE}{\hat\bE}
\newcommand{\hP}{{\hat\bP}}
\newcommand{\ha}{\hat\ga}
\newcommand{\hK}{\hat K}

\newcommand{\hmu}{\hat \mu}
\newcommand{\gann}{\gamma_{\rm ann}}

\newcommand{\gap}{{\rm gap}}

\newcommand{\sm}{{\raise0.3ex\hbox{$\scriptstyle \setminus$}}}

\newcommand{\beq}{\begin{equation}}
\newcommand{\eeq}{\end{equation}}
\newcommand{\ba}{\begin{aligned}}
\newcommand{\ea}{\end{aligned}}

\begin{document}

\title[Pinning on a renewal set]{The random pinning model with correlated disorder given by a renewal set}

\author{D.\ Cheliotis}
\address{National and Kapodistrian University of Athens, Department of Mathematics, Panepistimiopolis, 15784 Athens, Greece }
\email{dcheliotis@math.uoa.gr}

\author{Y.\ Chino}
\address{Mathematical Institute, Leiden University, P.O. Box 9512, 2300 RA Leiden, The Netherlands}
\email{y.chino@math.leidenuniv.nl}

\author{J.\ Poisat}
\address{Universit\'e Paris-Dauphine, CNRS, UMR [7534], CEREMADE, PSL Research University, 75016 Paris, France}
\email{poisat@ceremade.dauphine.fr}

\thanks{This work started when YC and DC visited Leiden University, where JP held a postdoc position. DC and JP were supported by ERC Advanced Grant 267356-VARIS of Frank den Hollander and YC was supported by the department of Mathematics of Hokkaido University. Part of this work was later carried out when DC visited Université Paris-Dauphine and JP visited University of Athens. The authors thank the respective institutions for their hospitality and support. JP acknowledges support from a PEPS grant ``Jeunes chercheurs'' of CNRS}

\begin{abstract}
We investigate the effect of correlated disorder on the localization transition undergone by a renewal sequence with loop exponent $\alpha>0$, when the correlated sequence is given by another independent renewal set with loop exponent $\hat \alpha >0$. Using the renewal structure of the disorder sequence, we compute the annealed critical point and exponent. Then, using a smoothing inequality for the quenched free energy and second moment estimates for the quenched partition function, combined with decoupling inequalities, we prove that in the case $\hat \alpha > 2$ (summable correlations), disorder is irrelevant if $\ga<1/2$ and relevant if $\ga >1/2$, which extends the Harris criterion for independent disorder. The case $\hat \alpha \in (1,2)$ (non-summable correlations) remains largely open, but we are able to prove that disorder is relevant for $\ga > 1/\hat \ga$, a condition that is expected to be non-optimal. Predictions on the criterion for disorder relevance in this case are discussed. Finally, the case $\hat\alpha \in (0,1)$ is somewhat special but treated for completeness: in this case, disorder has no effect on the quenched free energy, but the annealed model exhibits a phase transition.
\end{abstract}
\date{\today}
\maketitle
\section{Introduction}

\par The goal of this paper is to study the phase transition of the pinning model in presence of a correlated disorder sequence built out of a renewal sequence. We first present the general set-up of pinning models before introducing our specific model. For a review on pinning models, we refer to the three monographs~\cite{Gi2007, Gi2011, dH09} and references therein. In this paper we write $\bbN = \{1, 2, \ldots\}$ and $\bbN_0 = \{0, 1, 2,\ldots\}$.

\subsection{General set-up}

The pinning model provides a general mathematical framework for studying various physical phenomena such as the wetting transition of interfaces, DNA denaturation or (de)localization of a polymer along a defect line. This statistical-mechanical model is formulated in terms of a Markov chain $(S_n)_{n\in\bbN_0}$ which is given a reward/penalty $\omega_n$ (depending on the sign) when it returns to its initial state $0$ at time $n$. 

\par Let us denote by $\tau = (\tau_n)_{n\in\bbN_0}$ the sequence of return times to $0$, whose law is denoted by $\bP$. It is a renewal sequence starting at $\tau_0 = 0$, and we assume that the inter-arrival law satisfies
\begin{equation}\label{eq:defK_regvar}
K(n) := \bP(\tau_1 = n) = L(n) \, n^{- (1 + \ga)},\qquad \ga > 0,\quad n\in\bbN,
\end{equation}
where $L$ is a slowly varying function whose support is aperiodic, that is, $\gcd\{n \geq 1 \colon L(n) > 0\}= 1$. We also assume that the renewal process is recurrent, that is $\bP(\tau_1 < \infty) = \sum_{n\ge 1} K(n) = 1$ (otherwise it is said to be transient). By a slight abuse of notation, we shall use $\tau$ to refer to the set $\{\tau_k\}_{k\in\bbN_0}$ and write $\gd_n = \ind_{\{n\in\tau\}}$. Independently of $\tau$, we introduce a disorder sequence, that is a sequence of real valued random variables $\go = (\go_n)_{n\in \bbN_0}$ whose law is denoted by $\bbP$.

\par The object of interest is the sequence of Gibbs measures, also called polymer measures, defined by:
\begin{equation} \label{eq:defquemeas}
\frac{\dd \bP_{n,\gb,h}}{\dd \bP} = \frac{1}{Z_{n,\gb,h}} \exp \Big\{\sum_{k=1}^n (h + \gb \go_k)\gd_k \Big\} \gd_n, \quad n \in \bbN, \quad \gb \geq 0, \quad h \in \bbR,
\end{equation}
where
\begin{equation} \label{eq:defquepf}
Z_{n,\gb,h} = \bE \left( e^{\sum_{k=1}^n (h + \gb \go_k) \gd_k} \gd_n \right)
\end{equation}
is the quenched partition function, $h$ is called a pinning strength or chemical potential and $\gb$ is the inverse temperature.

The free energy of the model is defined by
\begin{equation}
F(\gb,h) = \lim_{n \to \infty} \frac{1}{n} \log Z_{n,\gb,h} \geq 0,
\end{equation}
where the limit holds $\bbP$-a.s.~and in $L^1(\bbP)$ under rather mild assumptions on $\go$, namely if $\go$ is a stationary and ergodic sequence of integrable random variables. Then the two phases of the model are the localized phase $\cL = \{(\gb, h): F(\gb, h)>0\}$, where the contact fraction $\partial_h F(\gb,h) = \lim_{n \to \infty} (1/n) \bE_{n,\gb,h}(|\tau \cap \{1,\ldots,n\}|)$ is positive, and the delocalized phase $\cD = \{(\gb, h): F(\gb, h)=0\}$, where it is zero.
 The two main features of the transition are the quenched critical point and the critical exponent:
\begin{equation}
h_c(\gb) = \inf\{h \colon F(\gb,h) > 0\} ,\quad \nu_q(\gb) = \lim_{h \searrow h_c(\gb)} \frac{\log F(\gb,h)}{\log (h - h_c(\gb))},
\end{equation}
when the limit exists. The critical curve separates the two phases whereas the critical exponent indicates how smooth the transition is between them.

\subsubsection{Disorder relevance}

One reason for the success of this model is the solvable nature of the homogeneous case, which corresponds to the choice $\gb = 0$ and which is treated in detail in \cite{Gi2007}. For the moment, we recall that $h_c(0) = 0$ and $\nu_{\hom}:= \nu_q(0) = \max(1,1/\ga)$, see Theorem 2.1 in \cite{Gi2007}. 

An important challenge in statistical mechanics is to understand the effect of quenched impurities or inhomogeneities in the interaction on the mechanism of the phase transition. This can be done by comparing the critical features of the quenched model to that of the annealed model, which is defined by
\begin{equation} \label{eq:defannmeas}
\frac{\dd \bP^a_{n,\gb,h}}{\dd \bP} = \frac{1}{Z^a_{n,\gb,h}} \bbE\left( \exp \Big\{ \sum_{k=1}^n (h + \gb \go_k) \gd_k \Big\} \gd_n \right), \quad n \in \bbN, \quad \gb \geq 0,\quad h \in \bbR,
\end{equation}
where
\begin{equation} \label{eq:defannpf}
Z^a_{n,\gb,h} = \bbE\bE\left(e^{\sum_{k=1}^n (h+ \gb \go_k)\gd_k} \gd_n\right) = \bbE (Z_{n, \gb, h})
\end{equation}
is the annealed partition function, and the annealed free energy is also defined by
\begin{equation}
F^a(\gb,h) = \lim_{n \to \infty} \frac{1}{n} \log Z^a_{n,\gb,h},
\end{equation}
when the limit exists. The annealed features are then
\begin{equation}
h^a_c(\gb) = \inf\{h \colon F^a(\gb,h) > 0\}, \quad \nu_a(\gb) = \lim_{h \searrow h_c^a(\gb)} \frac{\log F^a(\gb,h)}{\log (h - h_c^a(\gb))},
\end{equation}
when the limit exists (along the paper we may omit $\beta$ to lighten the notation, when there is no ambiguity). A simple application of Jensen's inequality leads to the following comparison
\begin{equation}
F(\gb,h) \leq F^a(\gb,h) \text{ for all } h\in \R, \gb\ge0, 
\end{equation}
and consequently
\begin{equation}
h_c^a(\gb) \leq h_c(\gb).
\end{equation}
If the annealed and quenched critical points or exponents differ at a given value of $\gb$ then disorder is said to be relevant for this value of $\beta$.

\subsubsection{The Harris criterion}

There has been a lot of studies on this problem in the past few years in the case when disorder is given by a sequence of i.i.d.~random variables with exponential moments (under this assumption the annealed model coincides with the homogeneous model after a suitable shift of $h$). All these works put on a firm mathematical ground the prediction known in the physics literature as the Harris criterion~\cite{H74}, which in this context states that disorder should be irrelevant if $\ga < 1/2$ (at least for small values of $\beta$) and relevant if $\ga > 1/2$. Several approaches have been used: direct estimates such as fractional moment and second moment estimates~\cite{AS06, AZ09, BL15, BCPSZ14, CadH13, DGLT09, GiT2006, GTL09}, martingale theory~\cite{L10}, variational techniques~\cite{CdH13} and more recently chaos expansions of the partition functions~\cite{CSZ16+,CSZ16,CTT16+}. The limiting case $\ga = 1/2$ has been the subject of a lot of controversies and has been fully answered only recently~\cite{BL15}. Finally, the full criterion for relevance (in the sense of critical point shift) reads
\begin{equation}
\forall \gb > 0, h_c(\gb) > h_c^a(\gb) \iff \sum_{n \geq 1} L(n)^{- 2} n^{2 (\ga - 1)} = \infty,
\end{equation}
that corresponds to the intersection of two independent copies of $\tau$ being recurrent.

\subsubsection{Correlated disorder: state of the art} 

The study of pinning models in correlated disorder is more recent, see~\cite{B13, B14, BL12, BP15, P13}. From a mathematical perspective, it is quite natural to try and understand how crucial the assumption of independence is for the basic properties of the polymer and in particular for the validity of the relevance criterion. Also, in several instances, the sequence of inhomogeneities may present more or less strong correlations: let us mention for instance the sequence of nucleotides which play the role of the disorder sequence in DNA denaturation~\cite{JPS06}. The main idea is that the relevance criterion should be modified only if the correlations are strong enough. Note that with correlated disorder, even the annealed model may not be trivial. Mainly two types of correlated disorder have been considered until now: correlated Gaussian disorder~\cite{B13, BP15, P13} and random environments with large attractive regions of sub-exponential decay, also referred to as {\it infinite disorder}~\cite{B14, BL12}.

\subsection{Scope of the paper}

The disorder sequence we consider is based on another renewal sequence $\htau$, independent of $\tau$, starting at the origin and whose law shall be denoted by $\hat\bP$. More specifically, we assume that if the Markov chain visits the origin at time $n$ then it is given a reward equal to one if $n \in \htau$, zero otherwise. We are therefore dealing with the following binary correlated disorder sequence:
\begin{equation}
\go_n = \hat\gd_n := \ind_{\{n \in \htau\}}, \qquad n \in \bbN_0.
\end{equation}
From now on, the inter-arrival laws of $\tau$ and $\htau$ satisfy
\begin{equation} \label{ReturnTimeDistributions}
K(n) := \bP(\tau_1 = n) \sim c_K \, n^{- (1 + \ga)},\quad \hK(n) := \hP(\htau_1 = n) = c_{\hK} \, n^{- (1 + \ha)}, \quad n \in \bbN
\end{equation}
with $\ga, \ha >0$, and
\begin{equation}
\mu := \bE(\tau_1),\qquad \hmu := \hE(\htau_1),
\end{equation}
which may be finite or infinite. Note that these definitions ensure aperiodicity for both renewal processes. In principle, the constants $c_K$ and $c_{\hK}$ may also be replaced by slowly varying functions, which would allow to include the special case $\ga \in \{0,1\}$ in the discussion, but we refrain from doing so for the sake of simplicity. Also, we write an equality in the definition of $\hat K(n)$ to ensure log-convexity. This technical condition is actually only needed for proving Theorem~\ref{thm:smoothing} (see Lemmas~\ref{lem:mon.pro} and \ref{lem:comp.tilt.shift}), which we actually believe to hold when the equality sign is replaced by the equivalent sign in the definition of $\hat K$ in \eqref{ReturnTimeDistributions}.

The definitions of the basic thermodynamical quantities is the same as in the previous section, except that $\bbP$ and $\bbE$ are replaced by $\hat\bP$ and $\hat\bE$. The condition that $n \in \tau$ in the definition of the polymer measures above could be removed, leading to the {\it free} versions. The versions with this condition are called the {\it pinned} versions.
It is a standard fact~\cite[Remark 1.2]{Gi2007} that this minor modification does not have any effect on the limiting free energies as defined in Propositions~\ref{pr:exist.ann.free.energy} and~\ref{pr:exist.que.free.energy}.

A first dichotomy arises:
\begin{itemize}
\item If $\ha < 1$, then the quantity in front of $\beta$ in the Hamiltonian is at most $|\htau \cap \{1,\ldots n\}|$, which is of order $n^{\ha}= o(n)$, and therefore disorder has no effect on the quenched free energy, which reduces to the homogeneous free energy. However the annealed model is non-trivial, so we include this case for completeness.\\

\item If $\ha > 1$, then (i) we may replace $\hP$ by its stationary version, denoted by $\hP_s$, under which the distribution of the increments $(\htau_{n + 1} - \htau_n)_{n \in \bbN_0}$ is the same as in $\hP$, whereas that of $\htau_0$ becomes $\{ \hP(\htau_1 > n) / \hmu\}_{n \in \bbN_0}$, see e.g.~\cite[Chapter V, Corollary 3.6]{A2003}  (Again, this does not affect the free energy, see Propositions~\ref{pr:exist.ann.free.energy} and~\ref{pr:exist.que.free.energy}) ; (ii) the correlation exponent of our environment is $\ha - 1$, since for $n > m$,
\begin{equation} \label{eq:speed_corr}
\begin{aligned}
\cov_{\hP_s}(\hat \gd_m, \hat \gd_n) &= \hat \bE_s(\hat \gd_m \hat \gd_n) - \hat \bE_s(\hat \gd_m) \hat \bE_s(\hat \gd_n)\\
&= \hP_s(m \in \htau) \Big( \hP_s(n \in \htau \mid m \in \htau) -  \hP_s(n \in \htau) \Big)\\
&= \frac{1}{\hmu} \Big( \hP(n - m \in \htau) -  \frac{1}{\hmu} \Big)\\
&\sim c (n - m)^{1 - \ha}, \qquad \text{ as } n - m \to \infty, 
\end{aligned}
\end{equation}
for some positive constant $c$. The latter can be deduced from the Renewal Theorem and the following renewal convergence estimate~\cite[Lemma 4]{Fr82}
\begin{equation} \label{eq:ren.conv.est}
\hP(n \in \htau) - \frac{1}{\hmu} \sim \frac{c_{\hK}}{\ha (\ha - 1) \hmu^2} \frac{1}{n^{\ha - 1}}, \qquad n \to \infty.
\end{equation}
\end{itemize}

\par Although our choice of disorder may seem at first quite specific, it is motivated by the following:
\begin{itemize}
\item By tuning the value of the parameter exponent $\ha$, one finds a whole spectrum of correlation exponents ranging from non-summable correlations to summable correlations, according to whether the sum $\sum_{n \geq 0} \cov_{\hP_s}(\hat \gd_0, \hat \gd_n)$ is infinite or finite. According to~\eqref{eq:speed_corr}, correlations are summable when $\ha > 2$ and non-summable when $\ha < 2$.
\item Our disorder sequence is bounded, therefore the annealed free energy is always finite, in contrast to the case of Gaussian variables with non-summable correlations~\cite{B13}.
\item The probability of observing a long sequence of ones decays exponentially in the length, which rules out the infinite disorder regime discussed in~\cite{B14}.
\item The renewal structure of the disorder sequence makes the study of the annealed model and decoupling inequalities more tractable.
\end{itemize}

\subsection{Summary of our results}

What we prove in this paper is the following: for the case $\ha > 2$ (summable correlations), disorder is irrelevant if $\ga < 1/2$ (both in the sense of critical points and critical exponents, at least for small values of $\beta$) and relevant if $\ga > 1/2$ (in the sense of critical exponents), which extends the Harris criterion for independent disorder. For the case $\ha \in (1,2)$ (non-summable correlations) all we are able to prove is that disorder is relevant when $\ga > 1/\ha$, a condition that we expect to be non-optimal. We discuss a list of predictions for disorder relevance in that case. Finally, in the case $\ha \in (0,1)$ disorder has no effect on the quenched free energy, but the annealed model exhibits a phase transition.

\subsection{Outline}

We present our results in Section 2. Section 2.1.~is dedicated to the annealed model and Section 2.2.~to the quenched model. The proofs for the annealed (respectively quenched) model are in Section 3 (resp.~Section 4). Results on renewal theory and homogeneous pinning are collected in the appendix.

\section{Results}

The intersection set of $\tau$ and $\htau$, which we denote by
\begin{equation} \label{eq:deftildetau}
\tilde\tau = \tau \cap \htau,
\end{equation}
will play a fundamental role in the sequel of the paper. Let us notice that it is itself a renewal starting at $\tilde\tau_0 = 0$. We denote its law by $\tilde\bP$ and write
\begin{equation} \label{eq:deftildedelta}
\tilde\gd_n := \ind_{\{n \in \tau \cap \htau\}} = \ind_{\{n \in \tau\}} \ind_{\{n \in \htau\}} = \gd_n \hat \gd_n, \qquad n \in \bbN_0.
\end{equation}
%
\subsection{Results on the annealed model}

We begin with the existence of the annealed free energy.
\begin{proposition} \label{pr:exist.ann.free.energy}
For all $\gb \geq 0$ and $h \in \bbR$, the annealed free energy
\begin{equation}
F^a(\gb,h) = \lim_{n \to \infty} \frac{1}{n} \log Z^a_{n,\gb,h}
\end{equation}
exists and it is finite and non-negative. The result still holds, without changing the value of the free energy, when $\hmu < \infty$ and $\hP$ is replaced by $\hP_s$.
\end{proposition}
\par The following basic properties of the annealed free energy are standard: the function $(\gb,h) \mapsto F^a(\gb,h)$ is convex, continuous and non-decreasing in both variables.
\subsubsection{An auxiliary function: the number of intersection points}

Before stating further results, we need to introduce an auxiliary function which will help us to characterize the annealed critical point. For $h \leq 0$, denote by $\bP_h$ the probability of the transient renewal process with $\tau_0 = 0$ and inter-arrival law
\begin{equation} \label{hTransRenewal}
K_h(n) = e^h K(n), \quad n \in \bbN, \qquad K_h(\infty) = 1 - e^h.
\end{equation}
We denote the corresponding expectation by $\bE_h$. The expected number of points in the renewal set $\tilde\tau$ (including $0$) under the law $\bP_{h} \times \hP$ is denoted by
\begin{equation}
\cI(h) := \bE_h \hE(|\tilde\tau|) \in [1,\infty].
\end{equation}
Note that
\begin{equation} \label{eq:cI}
\begin{aligned}
\cI(h) &= \sum_{n \in \bbN_0} \bP_h(n \in \tau) \hP(n \in \htau) = \sum_{n,k \in \bbN_0} e^{h k} \bP(\tau_k = n) \hP(n \in \htau)\\
&= \sum_{k \in \bbN_0} e^{h k} \bP \times \hP(\tau_k \in \htau).
\end{aligned}
\end{equation}
The function $\cI$ is finite and infinitely differentiable in $(- \infty, 0)$. It is also continuous in $(- \infty, 0]$, increasing and strictly convex. Its range is $[1,\cI(0)]$ with $\cI(0) = \bE \hE(|\tau \cap \htau|)$, which may be finite or infinite. It follows from Proposition~\ref{meanReturns} that
\begin{equation} \label{eq:rel_p_I}
p(h) := \bP_h \times \hP(\tilde\tau_1 < \infty) = 1 - \cI(h)^{-1}.
\end{equation}
\par Our next result provides an expression for the annealed critical curve involving the function $\cI$.
\begin{proposition} \label{pr:ann.crit.curve}
Let $\gb_0 = - \log p(0)$. The annealed critical curve is
\begin{equation} \label{anncc}
h^a_c(\gb) =
\begin{cases}
\cI^{-1} \Big(\frac{1}{1 - e^{-\gb}} \Big) & \text{ if } \gb > \gb_0, \\
0 & \text{ if } 0 \leq \gb \leq \gb_0.
\end{cases}
\end{equation}
\end{proposition}
\begin{remark} \label{rmk:signbeta0}
From~\eqref{eq:rel_p_I} we have that $\gb_0 = - \log(1 - \{\bE \hE(|\tilde\tau|)\}^{-1})$ is non-negative. Therefore, using Proposition~\ref{RMassFunctionProp}, we see that
\begin{equation} \label{eq:logp0}
\gb_0 \, 
\begin{cases}
>0 & \text{ if } \ga + \ha < 1,\\ 
=0 & \text{ if } \ga + \ha > 1.
\end{cases}
\end{equation}
\end{remark}
By the properties of $\cI$ we get that $\gb \mapsto h_c^a(\gb)$ is infinitely differentiable in $[0,\infty) \sm \{- \log p(0)\}$ and has negative derivative in $(- \log p(0),\infty)$. Moreover, $\gb \mapsto h_c^a(\gb)$ is concave because $(\gb,h) \mapsto F^a(\gb,h)$ is convex, see Figure~\ref{fig.ann.cc}.\\

The next two propositions provide the scaling behaviour of the annealed critical curve close to $\gb_0$.
\begin{proposition} \label{pr:ann.cc.scaling}
Suppose $\ga + \ha > 1$ (then $\gb_0 = 0$). There exists $c_a > 0$ such that
\begin{equation}
h_c^a(\gb) = - \frac{\gb}{\hmu} - c_a \gb^{\gann}[1 + o(1)], \quad \text{ as }\gb \searrow 0,
\end{equation}
where
\begin{equation} \label{eq:defgann}
\gann = 
\left\{
\begin{array}{lll}
1 + \Big[ \frac{\ha -1}{\ga \wedge 1}  \wedge 1 \Big]. & \mbox{ if } \ha > 1 \mbox{ and } \ha \neq 1 + \ga \wedge 1\\
\frac{\ga \wedge 1}{\ha -1 + \ga \wedge 1} & \mbox{ if } \ha < 1.
\end{array}
\right.
\end{equation}
If $\ha = 1 + \ga \wedge 1$, we get instead
\beq
h_c^a(\gb)  = - \frac{\gb}{\hmu} - c_a \gb^2 | \log \gb |[1 + o(1)].
\eeq
\end{proposition}
The first term $- \gb / \hmu$ simply accounts for the fact that our disorder sequence is not centered and that by the Renewal Theorem, $\lim_{n \to \infty} \hE(\hat \gd_n) = 1 / \hmu$. Note that by Jensen's inequality, $h_c^a(\gb) \leq - \gb / \hmu$, and this already gives that $c_a \geq 0$ in Proposition~\ref{pr:ann.cc.scaling}. If $\ha > 1 + \ga \wedge 1$, then $\gann = 2$, as in the i.i.d.~case, but if $\ha < 1 + \ga \wedge 1$, there is an anomalous scaling of the annealed critical curve. Moreover, if $\ha <1$ then $\hmu = \infty$ so the term $\gb/\hmu$ disappears and $\gann >1$ gives the first order term.
\begin{proposition} \label{pr:ann.cc.scaling2}
Suppose $\ga + \ha < 1$ (then $\gb_0 > 0$). As $\gb \searrow \gb_0$, there is a constant $c \in (- \infty,0)$ such that
\begin{equation}
h_c^a(\gb) \sim c\frac{(\gb - \gb_0)^{\gann}}{1+|\log(\gb-\gb_0)|\ind_{\{1-\hat \ga=2\ga\}}}, \quad \text{where} \quad
\gann = 1 \vee \frac{\ga}{1 - \ga - \ha}.
\end{equation}
\end{proposition}
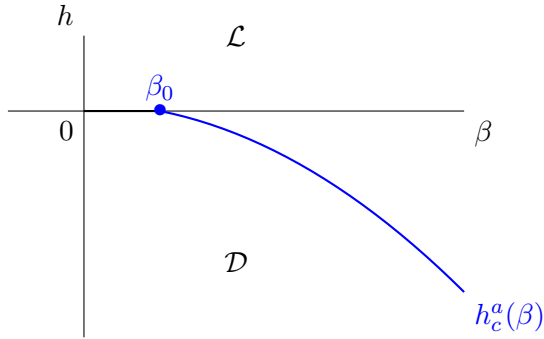
\begin{figure}
\begin{center}
\begin{tikzpicture}
\draw (-1,0)--(5,0) ;
\draw (0,1)--(0,-3) ;
\draw (0,0) node[below left]{$0$} ;
\draw (0,1) node[above left]{$h$} ;
\draw (5,0) node[below right]{$\gb$} ;
\draw (2,-2) node{$\mathcal{D}$} ;
\draw (2,1) node{$\mathcal{L}$} ;
\draw [
, thick ] (0,0)--(1,0) ;
\draw [blue] (1,0) node{$\bullet$} ;
\draw [blue] (1,0) node[above]{$\gb_0$};
\draw [blue, thick ]plot [domain=1:5] (\x, -0.1*\x*\x + 0.1) ;
\draw [blue] (5,-2.4) node[below right]{$h_c^a(\gb)$};
\end{tikzpicture}
\end{center}
\caption{\label{fig.ann.cc}Shape of the annealed critical curve (in blue). The critical point $\gb_0 = -\log p(0)$ and the slope at $\gb_0$ might be positive or equal to zero, depending on the values of $\ga$ and $\hat\ga$, see Remark~\ref{rmk:signbeta0} and Proposition~\ref{pr:ann.cc.scaling2}.}
\end{figure}
\par Our next result is about the order of the annealed phase transition.
\begin{proposition}[The annealed critical exponent] \label{ACExp}
Suppose $\ha > 0$. Let $\gb > 0$. There exists a constant $C = C(\gb) \in (0,\infty)$ such that
\begin{equation} \label{ACExpRelation}
(1/C) \leq \frac{F^a(\gb,h)}{(h - h_c^a(\gb))^{\nu_a(\gb)}} \leq C
\end{equation}
for all $0 < h - h_c^a(\gb) \leq 1$,
with
\begin{equation}
\nu_a(\gb) :=
\begin{cases}
\frac{1}{\ga_{\ef}} \vee 1 & \text{ if } \gb > \gb_0,\\
\frac{1}{\ga} \vee 1 & \text{ if } \, 0 \leq \gb \leq \gb_0,
\end{cases}
\end{equation}
where $\ga_\ef := \ga + (1 - \ha)_+$ and $(a)_+ := \max\{a,0\}$.
\end{proposition}
Therefore, the annealed critical exponent remains unchanged compared to the homogeneous case if $\ha > 1$, but is changed for large values of $\gb$ when $\ha < 1$ and $\ga < 1$.
\subsection{Results on the quenched model}

We start with the existence of the quenched free energy.
\begin{proposition} \label{pr:exist.que.free.energy}
For $\gb > 0$ and $h \in \bbR$ the sequence $\{(1/n) \log Z_{n,\gb,h}\}_{n \in \bbN}$ converges $\hP$-a.s. and in $L^1(\hP)$ to a non-negative constant $F(\gb,h)$ called the quenched free energy. Moreover, if $\hmu = \infty$ then $F(\gb,h) = F(0,h)$, and if $\hmu < \infty$ then the convergence still holds $\hP_s$-a.s. and in $L_1(\hP_s)$ (without changing the value of the free energy).
\end{proposition}
We are able to prove the following smoothing inequality:
\begin{theorem} \label{thm:smoothing}
Let $\ha > 1$
and $\gb > 0$. There exists a constant $C = C(\gb) \in (0,\infty)$ such that for $0 \leq h - h_c(\gb) \leq 1$,
\begin{equation}
F(\gb,h) \leq C (h - h_c(\gb))^{2 \wedge \ha}{(1 + |\log (h - h_c(\gb))| \ind_{\{\ha = 2\}})}.
\end{equation}
\end{theorem}
The exponent $2 \wedge \ha$ in the theorem above is not expected to be optimal, but in view of Proposition~\ref{ACExp}, this already tells us that disorder is relevant (in the sense that $\nu_q > \nu_a$) if $\ha > 2$ and $\ga > 1/2$, or if $\ha \in (1,2)$ and $\ha > 1/\ga$. This result extends the smoothing inequality obtained by Giacomin and Toninelli~\cite{GiT2006} in the i.i.d.~case.\\

We also prove the following result on disorder irrelevance:
\begin{theorem} \label{thm:irrelevance}
If $\ha > 2$ and $\ga < 1/2$ then disorder is irrelevant for $\gb$ small enough, meaning that $h_c(\gb) = h_c^a(\gb)$ and
\begin{equation}
\lim_{h \searrow h_c^a(\gb)} \frac{\log F(\gb,h)}{\log (h - h_c^a(\gb))} = \frac{1}{\ga}.
\end{equation}
\end{theorem}
To the best of our knowledge, such a result on disorder irrelevance (in both critical points and exponents) has not yet been proven for other instances of correlated disorder, e.g.~Gaussian disorder with summable correlations.\\

When $\hmu$ is infinite, the issue of critical point shift is settled thanks to Proposition~\ref{pr:ann.crit.curve} and Proposition~\ref{pr:exist.que.free.energy}, which tells us that $h_c(\gb) = 0$ for all $\gb \geq 0$. Thus we get that $h_c^a(\gb) = h_c(\gb)$ when $\gb \le - \log p(0)$ and $h_c^a(\gb) < h_c(\gb)$ when $\gb > - \log p(0)$.
The next proposition gives a condition under which $h_c^a(\gb) < h_c(\gb)$ for large $\gb$ when $\hmu$ is finite.
\begin{proposition} \label{LargeBetaRelevance}
If $\hmu < \infty$, then a sufficient condition under which $h_c^a(\gb) < h_c(\gb)$ for large enough values of $\gb$ is 
\begin{equation} \label{RelAtInfty}
- \log \bP \times \hP(\tau_1 \in \htau) > - \frac{1}{\bP \times \hP(\tau_1 \in \htau)} \sum_{n=1}^\infty \hP(n \in \htau) K(n) \log K(n).
\end{equation}
If we assume that $K$ is of the form $K_\ga(n) = c_\ga n^{-(1 + \ga)}$ for all $\ga > 0$ and $n \in \bbN$ where $c_\ga = 1 / \zeta(1 + \ga)$, then \eqref{RelAtInfty} is satisfied if $\ga$ is large enough.
\end{proposition}
Finally, our results on the issue of disorder relevance are summed up in Figure~\ref{fig.sum}.

\subsection{Discussion}

We collect here remarks about our results.\\

\noindent {\bf 1.} 
Note that when $\gb_0 > 0$, that is, when $\bE \hE(|\tau \cap \htau|) < \infty$, then for small $\gb$, the annealed critical exponent is the same as in the case when the renewal $\htau$ is absent. The reason behind this is that the reward $\gb$ given at each intersection point in $\tau \cap \hat\tau$ is too weak for $\htau$ to contribute to the free energy.\\

\noindent {\bf 2.} 
According to the Weinrib-Halperin criterion~\cite{WH83}, which aims to generalize the Harris criterion, disorder should be relevant if $\nu < \frac{2}{\xi \wedge 1}$ (at least for small disorder) and irrelevant if $\nu > \frac{2}{\xi \wedge 1}$, where $\nu$ is the critical exponent of the pure (homogeneous) system and $\xi$ is the correlation exponent of the environment. The application of this criterion to pinning models was introduced and discussed in~\cite{B13}. In our case, $\nu = (1/\ga) \vee 1$, $\xi = \ha - 1$ (assuming that $\ha > 1$) and the Harris criterion should not be changed if $\xi >1$, i.e.~$\ha > 2$, which is confirmed by Theorem~\ref{thm:smoothing} and Theorem~\ref{thm:irrelevance}. If $\ha \in (1,2)$, the criterion predicts that disorder is relevant (resp.~irrelevant) if $\ga > \frac{\ha - 1}{2}$ (resp.~$\ga < \frac{\ha - 1}{2}$). However, there is no clear evidence that this criterion gives the right prediction out of the Gaussian regime and it has actually been disproved in several examples~\cite{B13, B14}.\\

\noindent {\bf 3.} 
The recent work of Caravenna, Sun and Zygouras~\cite{CSZ16+} has opened a new perspective on the issue of disorder relevance. Their work examines conditions under which we may find a weak-coupling limit of quenched partition functions, with randomness surviving in the limit. More precisely, they determine conditions under which there exist sequences of parameters in the Hamiltonian (the coupling constants $h_n, \gb_n$ in our case) that converge to zero as the size of the system goes to infinity and such that the properly rescaled quenched partition function converges in distribution to a random limit, which is obtained in the form of a Wiener chaos expansion. In several instances, including the one of the pinning model in i.i.d.~environment, it was shown that these conditions coincide with those of disorder relevance. Applying this approach to our model leads to the following conjecture:

\begin{conjecture} \label{ConjRelevance}
Disorder is relevant for all $\gb>0$ (in the sense of critical point shift)~if
\begin{equation} \label{eq:conj1}
\ga > 1 - \frac{1}{\ha \wedge 2},
\end{equation}
in which case
\beq \label{eq:conj2}
\limsup_{\gb \to 0} \frac{\log (h_c(\gb) - h_c^a(\gb))}{\log \gb} = \frac{(\ga\wedge 1)(\ha\wedge 2)}{1-(\ha\wedge 2)(1-(\ga\wedge 1))}.
\eeq
\end{conjecture}

This problem will be attacked in a future work. The reason for the term $1/\ha$ in place of the usual $1/2$ when $\ha\in (1,2)$ is that the partial sums of our disorder sequence is in the domain of attraction of an $\ha$-stable law. More specifically:
\begin{equation}
\frac{1}{n^{1/\ha}} \sum_{k=1}^n (\hat \gd_k - 1/\hmu) \longrightarrow \ha\text{-stable law},\quad \text{as\ } n\to\infty,\quad \ha\in(1,2).
\end{equation}
Therefore we expect that white noise is replaced by a Levy noise in the weak-coupling limit of the quenched partition function. Note that \eqref{eq:conj1} and \eqref{eq:conj2} coincide with the case of i.i.d. disorder when $\ha>2$, that is the summable correlation scenario. Finally, another reason to believe in this conjecture is that the chaos expansion approach gives the right prediction for a pinning model in i.i.d.~$\gamma$-stable environment ($1 < \gamma < 2$), which has been studied recently by Lacoin and Sohier~\cite{LS16+}. There, it has been proved that disorder is relevant (resp.~irrelevant) if $\ga > 1 - 1/\gamma$ (resp.~$\ga < 1 - 1/\gamma$), which is to be compared to our conjecture.\\

\noindent {\bf 4.} 
The picture that has emerged for the moment regarding disorder relevance for this model can be summed up in the following exponent diagram, see Figure~\ref{fig.sum}.
\begin{figure}[htbp]
\begin{center}
\begin{tikzpicture}[xscale =3, yscale=2]
\fill[fill=blue!20]
(1/2,2)
-- plot [domain=1/2:1] (\x, {1/\x})
-- (0,1) -- (1,0) -- (2,0) -- (2,3) --(1/2,3) --(1/2,2);
\fill[fill=yellow!20] (0,2) rectangle (1/2,3);
\fill[fill=yellow!20] (0,1)--(1,0)--(0,0);
\draw [->](0,0) node[below]{$0$}--(1/2,0) node[below]{$1/2$}--(1,0) node[below]{$1$}--(2,0) node[below right]{$\ga$};
\draw [->](0,0) --(0,1) node[below left]{$1$}--(0,2) node[below left]{$2$}--(0,3) node[above left]{$\ha$};
\draw (0,1)--(1,0);
\draw [dotted ] plot [domain=0:1/2] (\x, 1+2*\x);
\draw [dashed] plot [domain=0:1/2] (\x, {1/(1-\x)});
\draw [densely dashed] (0,2) -- (2,2);
\draw [densely dashed] (0,1) -- (2,1);
\end{tikzpicture}
\end{center}
\caption{\label{fig.sum}
Disorder relevance/irrelevance in the exponent diagram.
}
\end{figure}
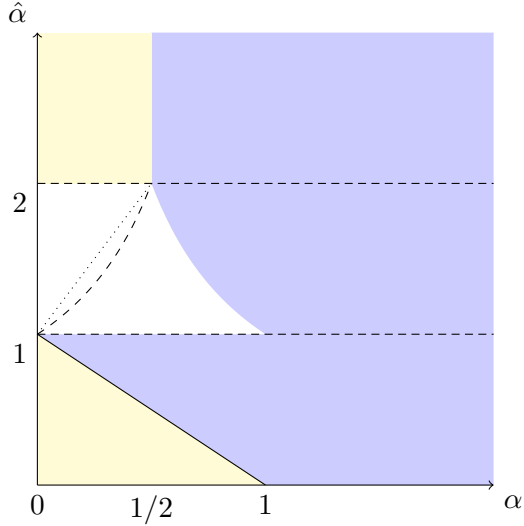
\begin{itemize}
\item The blue area is where we have proven relevance for small $\gb$.
\item In the region $\ha < 1$ we have relevance in the blue area because the quenched critical curve is trivially 0 while the annealed is strictly negative.
\item In the blue region with $\ha > 1$, we have relevance due to smoothing, see Theorem~\ref{thm:smoothing}. We do not know yet whether the critical points differ but we conjecture that they do so (see Conjecture \ref{ConjRelevance} above).
\item In the yellow triangle we have irrelevance because there also the annealed critical curve is 0 for small $\gb$ and the critical exponents agree. 
\item In the yellow part with $\ha > 2$, we have irrelevance due to Theorem~\ref{thm:irrelevance}.
\item The dashed line marks the border of relevance/irrelevance according to the chaos expansion heuristics when $\ga\in(0,1)$ and $\ha\in(1,2)$, see Conjecture~\ref{ConjRelevance}.
\item The dotted line marks the border of relevance/irrelevance according to the Harris-Weinrib-Helperin criterion when $\ga\in(0,1)$ and $\ha\in(1,2)$, see Item 2 in Discussion. 
\end{itemize} 

\noindent {\bf 5.}
Finally, let us mention the recent work of Alexander and Berger~\cite{A17_1} who also consider a pinning model with disorder built out of a renewal sequence. Even if they may look similar, the model studied in~\cite{A17_1} and the one considered in this paper are actually different in spirit. Indeed, in~\cite{A17_1} all the interactions up to the $n$-th renewal point of the disorder renewal (denoted here by $\htau_n$) are taken into account, and the only parameter is the inverse temperature $\gb$ (no pinning strength $h$). As a consequence, the results obtained therein are also quite different as for instance, the critical line deciding disorder relevance is at $\ga + \hat \ga = 1$. However, we do not exclude that the two models are related. For instance, the line $\ga + \hat \ga = 1$ also appears in Remark~\ref{rmk:signbeta0} above and, incidentally, in Proposition~\ref{ACExp} (see also Figure \ref{fig.sum}). 

\section{Proof of the annealed results}

The main idea is that the annealed model can be viewed as a homogeneous pinning model for the intersection renewal $\tilde\tau = \tau \cap \htau$ after the law of $\tau$ has been tilted.

\subsection{Existence of the free energy}
\begin{proof}[Proof of Proposition \ref{pr:exist.ann.free.energy}]
We use standard techniques, see the proof of Lemma~3.5 in~\cite{Gi2011}.  Let us first introduce the {\it fully-pinned}  annealed partition function
\begin{equation} \label{eq:fullpinned}
Z^{a, c}_{n,\gb,h} := \hat \bE \bE \left( e^{\sum_{k=1}^n (h + \gb \hat \gd_k) \gd_k} \gd_n \hat \gd_n \right).
\end{equation}
We shall write $Z^{a, c}_{n,\gb,h}(A)$ when the expectation above is restricted to the event $A$ and use the same convention for other versions of the partition function appearing in the proof.
By super-additivity, the sequence $\{n^{-1} \log Z_{n,h,\gb}^{a,c}\}_{n \in \bbN}$ converges to the limit
\begin{equation} \label{eq:facsup}
F^{a, c}(\gb,h) := \sup_{n \in \bbN} \{n^{-1} \log Z^{a,c}_{n,\gb,h}\}.
\end{equation}
From the bounds
\begin{equation}
e^{|h + \gb| n} \geq Z^{a,c}_{n,\gb,h} \geq Z^{a,c}_{n,\gb,h}(\tau_1 = \htau_1 = n) = e^{h + \gb} K(n) \hK(n),
\end{equation}
we get that $F^{a,c}(\gb,h) \in [0,\infty)$. Let us now prove that $F^a(\gb,h)$ exists and that $F^a(\gb,h) = F^{a,c}(\gb,h)$. Let $R = \sup \{k \leq n \colon k\in \tilde\tau\}$ be the last point in $\tilde\tau$ before $n$. Then,
\begin{equation}
\begin{aligned}
Z^{a}_{n,\gb,h} &= \hat \bE \bE \left( e^{\sum_{k=1}^n (h + \gb \hat \gd_k) \gd_k} \gd_n \right)\\
&= \sum_{r=0}^n \hat \bE \bE \left( e^{\sum_{k=1}^r (h + \gb \hat \gd_k) \gd_k} e^{\sum_{k=r+1}^n (h + \gb \hat \gd_k) \gd_k} \ind_{\{R = r\}} \gd_n \right)\\
&= \sum_{r=0}^n \hat \bE \bE \left( e^{\sum_{k=1}^r (h + \gb \hat \gd_k) \gd_k} \gd_{r} \hat \gd_r \right) \hat \bE \bE \left( e^{\sum_{k=r+1}^n (h + \gb \hat \gd_k) \gd_k} \ind_{\{\tilde{\tau} \cap [r+1,n] = \emptyset\}} \gd_n \mid r \in \tilde\tau \right)\\
&= \sum_{r=0}^n Z_{r,\gb,h}^{a,c} \hat \bE \bE \left( e^{h \sum_{k=r+1}^n \gd_k}\ind_{\{\tilde{\tau} \cap [r+1,n] = \emptyset\}} \gd_n \mid r \in \tilde\tau \right).
\end{aligned}
\end{equation}
From~\eqref{eq:facsup} we know that $Z_{r,\gb,h}^{a,c} \leq e^{rF^{a,c}(\gb,h)}$. Moreover,
\begin{equation}
\begin{aligned}
&\hat \bE \bE \left( e^{h \sum_{k=r+1}^n \gd_k} \ind_{\{\tau \cap \htau \cap [r+1,n] = \emptyset\}} \gd_n \mid r \in \tilde\tau \right) \leq \hat \bE \bE \left( e^{h \sum_{k=r+1}^n \gd_k} \gd_n \mid r \in \tilde\tau \right)\\
&= Z^{a,c}_{n-r,0,h} \frac{1}{\hP(n - r \in \htau)} \leq e^{(n - r) F^{a,c}(0,h)} \frac{1}{\hP(n - r \in \htau)}\\
&\leq e^{(n - r)F^{a,c}(\gb,h)} \frac{1}{\hP(n - r \in \htau)}.
\end{aligned}
\end{equation}
Thus,
\begin{equation} \label{eq:zna_ub}
Z^a_{n,\gb,h} \leq \sum_{r=0}^n e^{r F^{a,c}(\gb,h)} e^{(n - r) F^{a,c}(\gb,h)} \frac{1}{\hP(n - r \in \htau)} = e^{n F^{a,c}(\gb,h)} \sum_{r=0}^n \frac{1}{\hP(r \in \htau)}.
\end{equation}
By Proposition~\ref{RMassFunctionProp}, the last sum increases polynomially in $n$. Combining this inequality with $Z^{a,c}_{n,\gb,h} \leq Z^a_{n,\gb,h}$, we get that the free energy $F^a(\gb,h)$ exists and equals $F^{a,c}(\gb,h)$.

We now prove the second part of the result, namely that the limit for $\hP_s$ is the same as for $\hP$. Suppose that $\hmu < \infty$, and define $Z_n^{a,s} = \bE \hE_s \Big( \exp \Big\{ \sum_{k=1}^n (h + \gb \hat \gd_k) \gd_k \Big\} \gd_n \Big)$ (we temporarily remove $\gb$ and $h$, for conciseness). By restricting the expectation to the event $\{0 \in \htau\}$, we get on the one hand $Z_n^{a,s} \geq Z_n^a / \hmu$. On the other hand, by decomposing on the value of $\tilde\tau_1$, we obtain
\begin{equation} \label{eq:znas_ub}
\begin{aligned}
Z_n^{a,s} &= Z_n^{a,s}(\tilde\tau_1 > n) + \sum_{k=1}^n Z_n^{a,s}(\tilde\tau_1 = k)\\
&\leq e^{n (F(0,h) + o(1))} + \sum_{k=1}^n e^{k F(0,h)} Z^a_{n - k}.
\end{aligned}
\end{equation}
To go from the first to second line, we used the Markov property at $\tilde\tau_1$ and the fact that $\gd_k \hat \gd_k = 0$ for all $k < \tilde\tau_1$. Combining \eqref{eq:znas_ub}, \eqref{eq:zna_ub}, and using that $F(0,h) \leq F^a(\gb,h)$, we get that
\begin{equation}
\limsup_{n \to \infty} \frac{1}{n} \log Z_n^{a,s} \leq F^a(\gb,h),
\end{equation}
hence the result.
\end{proof} 

\subsection{Annealed critical curve}

This subsection is organized as follows: we start with Lemma~\ref{lem:acc} below, which we use to prove Proposition~\ref{pr:ann.crit.curve}. From Lemma~\ref{lem:asympt.0} we get Proposition~\ref{pr:ann.cc.scaling2} and Lemma~\ref{lem:asympt.I}, which in turn yields Proposition~\ref{pr:ann.cc.scaling}.
\begin{lemma} \label{lem:acc}
Let $h \leq 0$. Then, $F^a(\gb,h) = 0$ if and only if $\gb \leq - \log p(h)$.
\end{lemma}
\begin{proof}[Proof of Lemma \ref{lem:acc}]
Let $h \leq 0$. We know from the proof of Proposition~\ref{pr:exist.ann.free.energy} that $F^a(\gb,h)$ is the limiting free energy of the fully-pinned partition function in \eqref{eq:fullpinned}, which we may rewrite, using \eqref{hTransRenewal} and \eqref{eq:deftildedelta}, as
\begin{equation} \label{eq:equ.znac}
Z^{a,c}_{n,\gb,h} = \hE \bE_h \Big( e^{\gb \sum_{k=1}^n \tilde \gd_k} \tilde \gd_n \Big),
\end{equation}
which is the partition function of the homogenous pinning model with reward $\gb$ for the renewal $\tilde\tau$ under the law $\bP_h \times \hP$. It now follows from a standard fact about homogeneous pinning models (see \cite[Section 1.2.2, Equation (1.26)]{Gi2007}) that its critical point (as $\gb$ varies and $h$ is fixed) is at
\begin{equation}
\gb = - \log \bP_h \hP(\tilde \tau_1 < \infty) = - \log p(h), 
\end{equation}
see \eqref{eq:rel_p_I}.
\end{proof}

\begin{proof}[Proof of Proposition \ref{pr:ann.crit.curve}]
We distinguish two cases. Suppose first that $\gb \leq \gb_0 = - \log p(0)$. On the one hand, if $h \leq 0$, then $F^a(\gb,h) \leq F^a(\gb,0) = 0$, by Lemma~\ref{lem:acc}. On the other hand, if $h > 0$ then $F^a(\gb,h) \geq F^a(0,h) > 0$, by \cite[Theorem 2.7]{Gi2007}, since $\tau$ is recurrent. Thus, $h_c^a(\gb) = 0$ in this case. Suppose now that $\gb > \gb_0$. If $h \geq 0$, $F^a(\gb,h) > 0$ for the same reason as above, so we restrict to $h < 0$. Then, by Lemma~\ref{lem:acc}, $F^a(\gb,h) = 0$ if and only if $\gb \leq - \log p(h)$, that is $\cI(h) \geq (1 - e^{- \gb})^{-1}$ (recall \eqref{eq:rel_p_I} and the lines above). Since $\gb < \gb_0$, we have $\cI(0)=(1 - e^{- \gb_0})^{-1}> (1 - e^{- \gb})^{-1} > 1$, that is,  $(1 - e^{- \gb})^{-1}$ is in the range of $\cI$, so we get
\begin{equation}
\gb \leq - \log p(h) \quad \text{if and only if } \quad h \geq \cI^{-1} \Big( \frac{1}{1 - e^{- \gb}} \Big),
\end{equation}
hence the result.
\end{proof}

\color{black}

\smallskip

\begin{proof}[Proof of Proposition~\ref{pr:ann.cc.scaling}] Let us first consider the case $\ha > 1$ and $\ha \neq 1 + \ga \wedge 1$, and write
\begin{equation}\label{eq:asymphca}
h_c^a(\gb) = - \frac{\gb}{\hmu} (1 + \gep_\gb) \qquad \mbox{ with } \lim_{\gb\downarrow0} \gep_\gb = 0.
\end{equation}
Note that $\ha > 1$ implies $\gb_0 = 0$, by \eqref{eq:logp0}. Therefore, from Proposition~\ref{pr:ann.crit.curve}, we get on the one hand
\begin{equation} \label{eq:as.hca1}
\cI(h_c^a(\gb)) = \frac{1}{1 - e^{- \gb}} = \frac{1}{\gb} \Big( 1 + \frac{1}{2} \gb + o(\gb) \Big), \qquad \gb \downarrow 0,
\end{equation}
and on the other hand, from Lemma~\ref{lem:asympt.I} below and \eqref{eq:asymphca}, as $\gb \downarrow 0$,
\begin{equation} \label{eq:as.hca2}
\begin{aligned}
\cI(h_c^a(\gb)) &= \frac{1}{\hmu} \frac{1}{1 - e^{h_c^a(\gb)}} + c \, (- h_c^a(\gb))^{\gann - 2} [1 + o(1)]\\
&= \frac{1}{\gb} \Big( 1 - \gep_\gb [1 + o(1)] + \frac{1}{2 \hmu} \gb [1 + o(1)] + c \, \gb^{\gann - 1} [1 + o(1)] \Big),
\end{aligned}
\end{equation}
where $c$ is a positive constant that may change from line to line and $\gann$ is defined as in \eqref{eq:defgann}.  The result follows by identifying the right-hand sides in \eqref{eq:as.hca1} and \eqref{eq:as.hca2}. Indeed we get $\gep_\gb \sim c \gb^{\gann - 1}$ when $\gann<2$ and $\gep_\gb \sim (c-\tfrac12 (1-1/\hmu)) \gb$ when $\gann =2$. Note that in the latter case the constant $c-\tfrac12 (1-1/\hmu)$ is indeed positive by Lemma~\ref{lem:asympt.I}.
If $\ha = 1 + \ga \wedge 1$, the same method leads us to $\gep_\gb \sim c \gb |\log \gb|$, which proves our claim.
Finally, the case $\ha < 1$ is easier. Indeed, \eqref{eq:as.hca1} gives $\cI(h_c^a(\gb)) \sim 1/\gb$ and the result follows from Lemma~\ref{lem:asympt.I}. 
\end{proof}

\begin{proof}[Proof of Proposition~\ref{pr:ann.cc.scaling2}] 
Since $\gb_0 > 0$ we have by Proposition~\ref{pr:ann.crit.curve}
\beq\label{eq:ann.cc2a}
 \cI(h_c^a(\gb_0)) - \cI(h_c^a(\gb_0 + \gep)) \sim \frac{e^{-\gb_0}}{(1-e^{-\gb_0})^2} \gep, \qquad \gep\to0.
\eeq
Moreover, for some positive constant $c>0$,
\begin{equation}\label{eq:ann.cc2b}
\cI(0) - \cI(h) = \sum_{k\in\bbN_0} (1- e^{h k}) \bE \hP(\tau_k \in \htau) \sim c |h|^{1\wedge \frac{1-\ha-\ga}{\ga}}(1+|\log |h||\ind_{\{1-\hat \ga=2\ga\}})
\end{equation}
as  $h\to 0^-$.
The equivalence above follows from Lemma  \ref{seriesAsymptotics}. 
We get the final result by noting that $h_c^a(\gb_0)=0$ and combining \eqref{eq:ann.cc2a} and \eqref{eq:ann.cc2b}.
\end{proof}

\begin{lemma} \label{lem:asympt.0}
(i) If $\ha > 1$ then, as $k \to \infty$,
\begin{equation}
\bE \hP (\tau_k \in \htau) - \frac{1}{\hmu} \sim \frac{c_{\hK}}{\hmu^2 \ha (\ha - 1)} \left\{
\begin{array}{ll}
\mu^{1 - \ha} k^{1 - \ha} &\quad \mbox{ if } \ga > 1,\\
\bE(X_\ga^{1 - \ha}) k^{\frac{1 - \ha}{\ga}} & \mbox{ if } \ga \in (0,1),
\end{array}
\right.
\end{equation}
where $X_\ga$ is an $\ga$-stable random variable totally skewed to the right, with scale parameter $\gs > 0$ depending on the distribution of $\tau_1$ and shift parameter $0$ (see relation \eqref{eq:stable.cf} in the appendix for a reminder of these terms).\\
(ii) If $\ha \in (0,1)$,
\begin{equation}
\bE \hP (\tau_k \in \htau) \sim \frac{C_{\ha}}{c_{\hK}} \left\{
\begin{array}{ll}
\mu^{\ha - 1} k^{\ha - 1} &\quad \mbox{ if } \ga > 1,\\
\bE(X_\ga^{\ha - 1}) k^{\frac{\ha - 1}{\ga}} & \mbox{ if } \ga \in (0,1),
\end{array}
\right.
\end{equation}
where $C_{\ha}$ is as in Proposition~\ref{RMassFunctionProp}.
\end{lemma}

\begin{proof}[Proof of Lemma~\ref{lem:asympt.0}]
(i) From the renewal convergence estimates in~\cite[Lemma~4]{Fr82}, for $\ha > 1$, we get
\begin{equation}
\hP(n \in \htau) - \frac{1}{\hmu} \sim \frac{c_{\hK}}{\hmu^2 \ha (\ha - 1)} n^{1 - \ha}, \quad n \to \infty,
\end{equation}
so we have $\bP$-a.s.,
\begin{equation}
\hP(\tau_k \in \htau) - \frac{1}{\hmu} \sim \frac{c_{\hK}}{\hmu^2 \ha (\ha - 1)} \tau_k^{1 - \ha}, \quad k \to \infty,
\end{equation}
Since $\tau_k \geq k$, we may take the expectation in the line above and write
\begin{equation}
\bE \hP(\tau_k \in \htau) - \frac{1}{\hmu} \sim \frac{c_{\hK}}{\hmu^2 \ha (\ha - 1)} \bE(\tau_k^{1 - \ha}), \quad k \to \infty,
\end{equation}
and we may conclude the proof with Lemma~\ref{lem:conv.power.tauk}.\\
(ii) From Proposition~\ref{RMassFunctionProp}, we have $\bP$-a.s, if $\ha \in (0,1)$,
\begin{equation}
\hP(\tau_k \in \htau) \sim \frac{C_{\ha}}{c_{\hK}} \tau_k^{\ha - 1},\quad k \to \infty,
\end{equation}
and with the same argument as in (i),
\begin{equation}
\bE \hP(\tau_k \in \htau) \sim \frac{C_{\ha}}{c_{\hK}} \bE(\tau_k^{\ha - 1}), \quad k \to \infty.
\end{equation}
We may conclude thanks to Lemma~\ref{lem:conv.power.tauk}.
\end{proof}

\begin{lemma} \label{lem:asympt.I}
Suppose that $\ga + \ha > 1$. As $h \uparrow 0^-$,
\begin{equation}
\cI(h) - \frac{1}{\hmu} \frac{1}{1 - e^h} \sim \left\{
\begin{array}{ll}
c|h|^{\frac{1-\ha}{\ga\wedge 1} - 1} & \mbox{if } \ha \in (1-\ga\wedge 1, 1),\\
c|h|^{\frac{\ha - 1}{\ga\wedge 1} - 1} & \mbox{if } \ha \in (1, \ga \wedge 1+1),\\
c|\log |h|| & \mbox{if } \ha = \ga \wedge 1+1,\\
c & \mbox{if } \ha > \ga \wedge 1+1,
\end{array}
\right.
\end{equation}
where $c$ is a positive constant (note that in the first case $\hmu = \infty$ and the left-hand side is simply $\cI(h)$). Moreover, in the case $ \ha > \ga \wedge 1+1$, the constant $c$ satisfies $c>(1-\hat \mu^{-1})/2$.

\end{lemma}

\begin{proof}[Proof of Lemma~\ref{lem:asympt.I}] 
Recall \eqref{eq:cI}. For $h < 0$, write
\begin{equation}
\cI(h) - \frac{1}{\hmu} \frac{1}{1 - e^h} =  \sum_{k\in\bbN_0} e^{h k} \bE \Big[ \hP(\tau_k \in \htau) - \frac{1}{\hmu} \Big].
\end{equation}
Then, in the first three cases, the result follows from Lemma~\ref{lem:asympt.0} and the standard Tauberian arguments recalled in Lemma~\ref{seriesAsymptotics} (iii).
In the fourth case, we have
$$c=\sum_{k\in\bbN_0} \bE \Big[ \hP(\tau_k \in \htau) - \frac{1}{\hmu} \Big].$$
Note that the set $\sigma:=\{k\in \bbN_0: \tau_k\in \htau\}$ defines a renewal process. Call $\gs_1$ its first positive point. This has mean $\hat \mu$ because $\bP\times \hP(k\in \sigma)\to 1/\hat \mu$ as $k\to\infty$. Let us call $v\in(0, \infty]$ its variance. We now use Problem 19 in Chapter XIII of \cite{Feller68}, which needs a fix : we let the reader check that in the equations (12.1) and (12.2) therein, $u_0$ should be removed and the sums should start at $n=0$.
According to this, we have
$$c=\sum_{k\in\bbN_0} \bE \Big[ \hP(k \in \gs) - \frac{1}{\hmu} \Big]=\frac{v^2+\hat \mu^2-\hat \mu}{2\hat \mu^2},$$
so that $c-(1-\hat \mu^{-1})/2=v^2/(2\hat\mu^2)>0$.
\end{proof}

\subsection{The annealed critical exponent}

Recall the definition of $\bP_h$ in \eqref{hTransRenewal}. When $\gb > \gb_0 = - \log p(0)$, it will be useful for the proof of Proposition~\ref{ACExp}, which follows, to introduce the inter-arrival law
\begin{equation} \label{tildeRenewal}
\tilde K_\gb(n) := e^\gb \bP_{h_c^a(\gb)} \times \hP((\tau \cap \htau)_1 = n), \qquad n \geq 1.
\end{equation}
We denote by $\tilde \bE_\gb$ the expectation with respect to this law (note that for $\gb=0$, it coincides with the law of $\tilde\tau$ under $\bP \times \hP$). It is a by-product of the proof of Lemma~\ref{lem:acc} that $h_c^a(\gb)$ is chosen such that the sequence $(\tilde K_\gb(n))_{n\in\N}$ sums up to 1, and thus defines a recurrent renewal.
\begin{remark} \label{rmk:freeenergyhomopol}
Note that if $F_\gb$ is the free energy of the homopolymer with inter-arrival law distribution $n \mapsto \bP_{h_c^a(\gb)} \times \hP((\tau \cap \htau)_1 = n)$, then the homopolymer with inter-arrival law distribution $\tilde K_\gb$ and pinning reward $h \geq 0$ has a free energy equal to $F_\gb(\gb+h)$.
\end{remark}
Let us start with a uniform estimate on the mass renewal function under $\bP_h \times \hP$, that will be used in the proof of Proposition~\ref{ACExp} below.
\begin{lemma} \label{productAsymptotics}
For any compact subset $J$ of $(- \infty,0)$ there are constants $0 < c_1 < c_2$ such that 
\begin{equation} \label{RFAsymptotics}
c_1 \leq \bP_h \times \hP(n \in \tau \cap \htau) \, n^{1 + \ga_\ef} \leq c_2,
\end{equation}
where $\ga_\ef = \ga + (1 - \ha)_+$, for all $h \in J$ and $n \geq 1$.
\end{lemma}
\begin{proof}[Proof of Lemma~\ref{productAsymptotics}] First, the probability in the display equals $\bP_h(n \in \tau) \hP(n \in \htau)$. As a lower bound for $\bP_h(n \in \tau)$, we get
\begin{equation}
\bP_h(n \in \tau) \geq K_h(n) = e^h c_K[1 + o(1)] \, n^{- (1 + \ga)},
\end{equation}
while for an upper bound we let $h_0 = \sup J < 0$ and use the fact that $\bP(\tau_k = n) \leq k^c K(n)$ for all $n,k \geq 1$ and some constant $c > 0$ (see~\cite[Lemma A.5]{Gi2007}) to get
\begin{equation}
\bP_h(n \in \tau) = \sum_{k=1}^n e^{h k} \bP(\tau_k = n) \leq K(n) \sum_{k=1}^n e^{h_0 k} k^c, 
\end{equation}
for all $h \in J$. It is crucial here that the compact set $J$ does not include $0$ so that $h_0 < 0$. Combining these estimates with Proposition~\ref{RMassFunctionProp}, we get the result.
\end{proof} 

\begin{proof}[Proof of Proposition~\ref{ACExp}] 
We consider the two possible cases for $\beta$:

\smallskip

\noindent \textbf{Case 1:} Assume $\gb > \gb_0$. Then $h_c^a(\gb) < 0$ by Proposition~\ref{pr:ann.crit.curve}.

\smallskip

\textsc{Lower bound:} We bound the partition function from below, as follows:
\begin{equation}
\begin{aligned}
Z^a_{n,\gb,h_c^a + \gep} &\geq \bE \hE \left(e^{(h_c^a + \gep) \sum_{k=1}^n \gd_k + \gb \sum_{k=1}^n \gd_k \hat \gd_k} \gd_n \hat \gd_n \right) = \bE_{h_c^a} \hE \left( e^{\gep \sum_{k=1}^n \gd_k + \gb \sum_{k=1}^n \gd_k \hat \gd_k} \gd_n \hat \gd_n \right)\\
&\geq \bE_{h_c^a} \hE \left( e^{(\gb + \gep) \sum_{k=1}^n \gd_k \hat \gd_k} \gd_n \hat \gd_n \right) = \tilde \bE_\gb \left( e^{\gep \sum_{k=1}^n \tilde \gd_k} \tilde \gd_n \right),
\end{aligned}
\end{equation}
where in the last equality we use \eqref{tildeRenewal}. By Lemma~\ref{productAsymptotics} applied to the set $J = \{h_c^a(\gb)\}$, we know that the mass renewal function $n \mapsto \bP_{h_c^a(\gb)} \times \hP(n \in \tilde\tau)$ satisfies \eqref{RFAsymptotics}. The lower bound then follows if one recalls Remark~\ref{rmk:freeenergyhomopol} and use Lemma~\ref{FreeEnergyU}, where the singleton $\{\gb\}$ and the renewal $n \mapsto \bP_{h_c^a(\gb)} \times \hP(\tau \cap \htau)_1 = n)$ play the role of the compact set $I$ and the renewal $K_\gga$ therein.

\smallskip

\textsc{Upper bound:} Pick $\gb_1 \in (\gb_0,\gb)$, and let $\gep > 0$ be small enough so that $h_c^a(\gb) + \gep < h_c^a(\gb_1) < 0$. By continuity of $h_c^a$, there is $\gb_\gep \in (\gb_1,\gb)$ so that $h_c^a(\gb) + \gep = h_c^a(\gb_\gep)$. Moreover, by the mean value theorem, there is $\xi_\gep \in (\gb_\gep,\gb)$ with $h_c^a(\gb) - h_c^a(\gb_\gep) = (h_c^a)'(\xi_\gep) (\gb - \gb_\gep)$. Thus, $\gb - \gb_\gep = c(\gb,\gep) \, \gep$ with $c(\gb,\gep) = - 1/(h_c^a)'(\xi_\gep)$, which converges to $c(\gb,0) = - 1/(h_c^a)'(\gb) > 0$ as $\gep \to 0^+$, by Proposition~\ref{pr:ann.crit.curve} and the regularity properties of $\cI$. Then, since $h_c^a(\gb) + \gep < 0$, (recall \eqref{eq:fullpinned})
\begin{equation}
\begin{aligned}
Z^{a,c}_{n,\gb,h_c^a + \gep} &= \bE \hE \left(e^{(h_c^a + \gep) \sum_{k=1}^n \gd_k + \gb \sum_{k=1}^n \gd_k \hat \gd_k} \gd_n \hat \gd_n \right) = \bE_{h_c^a + \gep} \hE \left( e^{\gb \sum_{k=1}^n \gd_k \hat \gd_k} \gd_n \hat \gd_n \right)\\
&= \bE_{h_c^a(\gb_\gep)} \hE \left( e^{\{\gb_\gep + c(\gb,\gep) \gep\} \sum_{k=1}^n \gd_k \hat \gd_k} \gd_n \hat \gd_n \right) = \tilde \bE_{\gb_\gep} \left( e^{c(\gb,\gep) \gep \sum_{k=1}^n \tilde \gd_k} \tilde \gd_n \right)\\
&\leq \tilde \bE_{\gb_\gep} \left( e^{2 c(\gb,0) \gep \sum_{k=1}^n \tilde \gd_k} \tilde \gd_n \right) \quad \text{for $\gep$ small enough,}
\end{aligned}
\end{equation}
where $\tilde \bE_{\gb_\gep}$ is the expectation with respect to the renewal defined in \eqref{tildeRenewal} above with $\gb_\gep$ in place of $\gb$. As in the lower bound part, the result follows by using Remark~\ref{rmk:freeenergyhomopol} and Lemma~\ref{FreeEnergyU} with $I = [\gb_1,\gb]$ and with the role of $K_\gga$ played by the law $n \mapsto \bP_{h_c^a(\gga)} \times \hP((\tau \cap \htau)_1 = n)$. The assumptions of the lemma are satisfied due to Lemma~\ref{productAsymptotics}, because $J := h_c^a(I)$ is a compact subset of $(- \infty,0)$ as $h_c^a$ is continuous, decreasing, with $h_c^a(\gb_1) < 0$.

\smallskip

\noindent \textbf{Case 2:} Assume $\gb \leq \gb_0$. Then $h_c^a(\gb) = 0$, by Proposition~\ref{pr:ann.crit.curve}. 

\smallskip

\textsc{Lower bound:} Since
\begin{equation}
Z^a_{n,\gb,\gep} \geq \bE \hE \left( e^{\gep \sum_{k=1}^n \gd_k} \gd_n \hat \gd_n \right) = \bE \left( e^{\gep \sum_{k=1}^n \gd_k} \gd_n \right) \hP(n \in \htau),
\end{equation}
the lower bound follows from standard results on homogeneous pinning, see~\cite[Theorem 2.1]{Gi2007}. 

\textsc{Upper bound:} We assume that $\gep \in (0,\gb)$. Pick $p > 1$ so that $p (\gb - \gep) \leq \gb_0$, and let $q > 1$ be defined by $p^{-1} + q^{-1} = 1$. Then, by H\"older's inequality,
\begin{equation}
\begin{aligned}
Z^{a,c}_{n,\gb,\gep} &= \bE \hE \left( e^{\gep \sum_{k=1}^n \gd_k + \gb \sum_{k=1}^n \gd_k \hat \gd_k} \gd_n \hat \gd_n \right) \leq \bE \hE \left( e^{2 \gep \sum_{k=1}^n \gd_k + (\gb - \gep) \sum_{k=1}^n \gd_k \hat \gd_k} \gd_n \hat \gd_n \right)\\
&\leq \bE \hE \left( e^{2 q \gep \sum_{k=1}^n \gd_k} \gd_n \right)^{1/q} \bE \hE \left( e^{p (\gb - \gep) \sum_{k=1}^n \gd_k \hat \gd_k} \gd_n \hat \gd_n \right)^{1/p}.
\end{aligned}
\end{equation}
Observe that the quantity $\bE \hE(e^{p (\gb - \gep) \sum_{k=1}^n \gd_k \hat \gd_k} \gd_n \hat \gd_n)$ is the partition function at $p (\gb - \gep)$ for the homopolymer defined by the renewal $\tau \cap \htau$, whose critical parameter is $\gb_0$. Thus, we obtain $F^a(\gb,\gep) \leq \frac{1}{q} F(0,2 q \gep)$ and the required bound follows again from~\cite[Theorem 2.1]{Gi2007}.
\end{proof}

\section{Proof of the quenched results}
\subsection{Existence of the free energy}

\begin{proof}[Proof of Proposition \ref{pr:exist.que.free.energy}]
When $\gb = 0$, the model reduces to the homogeneous pinning model, for which we know that the free energy $F(0,h)$ exists. Therefore we assume that $\gb > 0$, and consider two cases:

\textsc{Case 1}: Assume that $\hmu = \infty$.
\noindent Then,
\begin{equation}
\bE \left( e^{h \sum_{k=1}^n \gd_k} \gd_n \right) \leq Z_{n,\gb,h} \leq e^{\gb|\htau \cap \{1,\ldots,n\}|} \bE \left( e^{h \sum_{k=1}^n \gd_k} \gd_n \right).
\end{equation}
Since, by the Renewal Theorem, $|\htau \cap \{1,\ldots,n\}| / n \to 1/\hmu = 0$ as $n \to \infty$, almost surely and in $L^1(\hP)$, we get that $(n^{-1} \log Z_{n,\gb,h})_{n \in \bbN}$ converges to $F(0,h)$.

\textsc{Case 2} : Assume that $\hmu < \infty$. 
\noindent Suppose first that $\htau$ is distributed according to $\hP_s$, in which case we apply Kingman's subadditive ergodic theorem. Indeed, if we define (let us temporarily omit $\gb$ and $h$)
\begin{equation}\label{eq:def_alt_part_fct}
Z_{m,n} = \bE \left( e^{\sum_{k=m+1}^n (h + \gb \hat \gd_k) \gd_k} \gd_n \mid m \in \tau \right), \qquad n > m \geq 0,
\end{equation}
then, by restricting $Z_{0,n}$ to the event $\{m \in \tau\}$ and using the Markov property, we get
\begin{equation} \label{eq:zmn_superadd}
Z_{0,n} \geq Z_{0,m} Z_{m,n},
\end{equation}
so that the process $\{- \log Z_{0,n}: n \geq 0\}$ is sub-additive. The remaining assumptions of Kingman's subadditive ergodic theorem can be shown to be satisfied, see~\cite[Theorem 7.4.1]{Du10}) (it is crucial here that $\htau$ is stationary), and the claim of the proposition follows. Let us temporarily denote by $F_s(\gb,h)$ the quenched free energy in this case.

Suppose now that $\htau$ is distributed according to $\hP$. For $0 \leq m < n$, we define
\begin{equation}
\cZ_{m,n} = Z_{\htau_m,\htau_n} = \bE \left( e^{\sum_{k=\htau_m+1}^{\htau_n} (h + \gb \hat \gd_k) \gd_k} \gd_{\htau_n} \mid \htau_m \in \tau \right).
\end{equation}
Then, similarly to \eqref{eq:zmn_superadd},
\begin{equation}
\cZ_{0,n} \geq \cZ_{0,m} \cZ_{m,n},
\end{equation}
and one can check that the process $\{Y_{0,n} := - \log \cZ_{0,n}: n \geq 0\}$ satisfies the conditions of Kingman's subadditive ergodic theorem. Note that condition (iv) in that theorem~\cite[Theorem 7.4.1]{Du10} is satisfied because $\hat \bE (\log \cZ_{0,n}) \leq (|h| + \gb) \hmu n$ and $\log \cZ_{0,1} \geq h + \gb +  \log K(\htau_1)$, which is integrable w.r.t.~$\hP$. Thus, $\lim_{n \to \infty} \frac{1}{n} Y_{0,n}$ exists $\hP$-a.s. and in $L^1(\hP)$, and is non-random. Let us denote this limit by $\phi(\gb,h)$. Now, for any $n \in \bbN$, there exists a unique $k(n) \geq 0$ such that $n \in (\htau_{k(n)},\htau_{k(n)+1}]$. Therefore,
\begin{equation}
Z_{0,n} \geq Z_{0,\htau_{k(n)}} Z_{\htau_{k(n)},n} \geq Z_{0,\htau_{k(n)}} e^{h + \gb} K(n - \htau_{k(n)}),
\end{equation}
and with the same reasoning,
\begin{equation}
Z_{0,n} \leq e^{- (h + \gb)} Z_{0,\htau_{k(n)+1}} / K(\htau_{k(n)+1} - n).
\end{equation}
Using the fact that $(n^{-1} \log \htau_n)_{n \in \bbN}$ converges to zero $\hP$-a.s.~and in $L^1(\hP)$, in combination with the lower and upper bounds above, we obtain that $\lim_{n \to \infty} \frac{1}{n} \log Z_{0,n}$ exists $\hP$-a.s.~and in $L^1(\hP)$, and equals $F(\gb,h) = \phi(\gb,h) / \hat \bE(\htau_1).$

We now prove that $F_s(\gb,h) = F(\gb,h)$. Assume in the rest of the proof that $\htau$ is distributed according to $\hP_s$. In the following we shall temporarily write $Z_n$ with a superscript indicating on which environment we are considering the partition function. By imposing the first step of $\tau$ to be equal to $\htau_0$, we get
\begin{equation} \label{eq:lb_Znhtau}
Z_{n + \htau_0}^{\htau} \geq Z_{n}^{\htau - \htau_0} K(\htau_0) e^{\gb + h}, \qquad n \in \bbN,
\end{equation}
where $\htau - \htau_0 = \{\htau_i - \htau_0 \colon i \geq 0\}$. Since $\htau - \htau_0$ has law $\hP$, we get $F_s(\gb,h) \geq F(\gb,h)$. In the other direction, we get, since $|\tau \cap (0,\htau_0]| \leq \htau_0$,
\begin{equation} \label{eq:comb1}
Z_{n + \htau_0}^{\htau} \leq e^{\htau_0 (\gb + |h|)} \check Z_n^{\htau - \htau_0}, \qquad n \in \bbN,
\end{equation}
where $\check Z_n$ is the partition function obtained by replacing $\bP$ by a slightly modified law $\check \bP$ in which the distribution of the first inter-arrival $\tau_1$ is the one of the overshoot of $\tau$ with respect to $\htau_0$, namely
\begin{equation}
\check \bP(\tau_1 = \ell) = \bP(\inf \{n \geq \htau_0 \colon n \in \tau \} = \htau_0 + \ell), \qquad \ell \geq 0.
\end{equation}
Note that this distribution depends on the random variable $\htau_0$.
By decomposing on this first inter-arrival, we get
\begin{equation} \label{eq:comb2}
\begin{aligned}
\check Z_n^{\htau - \htau_0} &= \check \bP(\tau_1 > n) + \sum_{k=0}^n \check \bP(\tau_1 = k) Z_{n - k}^{\htau - \htau_0 -k}\\
&\leq 1 + e^{-h}\sum_{k=0}^n  K(k)^{-1} Z_n^{\htau - \htau_0}, \quad \text{ (with the convention $K(0)=1$ and $Z_0 = 1$)}
\end{aligned}
\end{equation}
where in the last inequality we have bounded the probabilities by one and used that $Z_n^{\htau - \htau_0} \ge K(k) e^{h} Z_{n - k}^{\htau - \htau_0 - k}$, similarly to \eqref{eq:lb_Znhtau}. Combining \eqref{eq:comb1} and \eqref{eq:comb2}, and using that $\sum_{k=1}^n K(k)^{-1}$ is only polynomially increasing in $n$, we may now compare the almost-sure limits of $(1/n) \log Z_{n + \htau_0}^{\htau}$ and $(1/n) \log Z_n^{\htau - \htau_0}$ and obtain that $F_s(\gb,h) \leq F(\gb,h)$. This completes the proof.
\end{proof}

\subsection{Smoothing inequality}
This section is devoted to the proof of Theorem~\ref{thm:smoothing}.
Our proof is inspired by the original work of Giacomin and Toninelli in the i.i.d.~disorder set-up~\cite{GiT2006}  and is based on a localization strategy, which for the polymer consists in hitting favorable regions of the environment, where disorder is tilted. The cost of finding such regions is given by the entropy estimate in Lemma~\ref{lem:spec.rel.ent.tilt}. Inspired by Caravenna and den Hollander~\cite{CadH13}, we also compare the free energy for a tilted disorder with that for a shifted disorder, see Lemma~\ref{lem:comp.tilt.shift}.\\

Let us start by considering the family of tilted disorder measures defined by
\begin{equation} \label{eq:def.tilt.meas}
\frac{\dd \hat \bP_{n,\gt}}{\dd \hat \bP} = \frac{e^{\gt \sum_{k=1}^n \hat \gd_k}}{\hat Z_{n,\gt}} \, \hat \gd_n, \qquad n \in \bbN, \quad \gt \in \bbR.
\end{equation}
Note that this is nothing else than a pinning measure for the disorder renewal and the normalizing constant in \eqref{eq:def.tilt.meas} is just the partition function of the homopolymer ($\gb = 0$) with pinning strength $\gt$ and underlying renewal $\hat\tau$. The corresponding relative entropy rate is defined as
\begin{equation} \label{eq:def.spec.rel.ent}
h_\infty(\gt) = \lim_{n \to \infty} \frac{1}{n} h(\hP_{n,\gt} | \hP) = \lim_{n \to \infty} \frac{1}{n} \hE_{n,\gt} \Big( \log \frac{\dd \hP_{n,\gt}}{\dd \hP} \Big).
\end{equation}
The proof of Lemma~\ref{lem:spec.rel.ent.tilt} below shows that this limit exists. Furthermore, it is non-negative as the limit of non-negative real numbers.\\

For the proof of Theorem~\ref{thm:smoothing} we will need three lemmas, which we now state and prove.
\begin{lemma}[Asymptotics of the relative entropy rate] \label{lem:spec.rel.ent.tilt}
There exists a constant $c \in (0,\infty)$ such that
\begin{equation}
h_\infty(\gt) = \gt \hat F'(\gt) - \hat F(\gt) \sim c \, \gt^{2 \wedge \ha}{(1+ |\log \gt| \ind_{\{\ha = 2\}})} \qquad \text{as } \gt \downarrow 0^+,
\end{equation}
where $\hat F(\gt)$ is the free energy of the homogeneous pinning model with parameter $\gt$ and renewal $\htau$.
\end{lemma}
\begin{proof}[Proof of Lemma~\ref{lem:spec.rel.ent.tilt}] A straightforward computation gives
\begin{equation} \label{eq:ent_n}
\frac{1}{n} h(\hP_{n,\gt} | \hP) = \frac{1}{n} \hE_{n,\gt} \Big( \log \frac{\dd \hP_{n,\gt}}{\dd \hP} \Big) = \gt \, \hE_{n,\gt} \Big( \frac{1}{n} \sum_{k=1}^n {\hat \gd_k} \Big) - \frac{1}{n} \log \hat Z_{n,\gt},
\end{equation}
and it is now a standard fact for homogeneous pinning models that the expectation above converges to $\hat F'(\gt)$, see~\cite[Section 2.4]{Gi2007}. Thus, as $n \to \infty$, the quantities in \eqref{eq:ent_n} converge to
\begin{equation}
\gt \hat F'(\gt) - \hat F(\gt) = \{\gt (\hat F'(\gt) - \hat F'(0))\} + \{\gt \hat F'(0) - \hat F(\gt)\}.
\end{equation}
A Tauberian analysis reveals that both terms are of order $\gt^{2 \wedge \ha}{(1+ |\log \gt| \ind_{\{\ha = 2\}})}$, as proved in Proposition~\ref{pr:asympt.hom.free.energy}, and the precise values of the constants in \eqref{eq:asymptoticsFh} show that the constant $c$ in Lemma~\ref{lem:spec.rel.ent.tilt} is indeed positive.
\end{proof}
Let $[n] = \{1,2,\ldots,n\}$. The probability distribution $\hP_{n,\gt}$ naturally induces a measure on $\{0,1\}^{[n]}$, which is a partially ordered set with the following order relation: for the configurations $\eta, \eta' \in \{0,1\}^{[n]}$, we write $\eta \leq \eta'$ if $\eta(x) \leq \eta'(x)$ for every $x \in[n]$. The law $\hP_{n,\gt}$ is called {\it monotone} 
 if 
\begin{equation} \label{RFCondProbability}
\hP_{n,\gt}(\hat \gd_k = 1 \mid \hat \gd([n] \sm \{k\}) = \eta([n] \sm \{k\}) ) \leq \hP_{n,\gt}(\hat \gd_k = 1 \mid \hat \gd([n] \sm \{k\}) = \eta'([n] \sm \{k\})),
\end{equation}
for every $k\in[n]$ and $\eta, \eta' \in \{0,1\}^{[n]}$ such that $\eta \leq \eta'$ (see Definition~4.9 in~\cite{GHM1999}) and both conditional probabilities are well-defined.

\begin{lemma}[Monotonicity property] \label{lem:mon.pro}
If the sequence $\{\hK(n)\}_{n \in \bbN}$ is log-convex\footnote{We say that a sequence of positive real numbers $\{u_n\}_{n \in \bbN}$ is log-convex if $u_n^2 \leq u_{n-1} u_{n+1}$  for all $n \geq 2$. It is equivalent to that 
$\frac{u_{n+m}}{u_{n+m'}} \leq \frac{u_m}{u_{m'}}$ for all $m, m',n\in \bbN$ with  $m \leq m'$.
}
 then the law $\hP_{n,\gt}$ is monotone for every $n \geq 1$.
\end{lemma}
We may prove that this is actually an equivalence, but we won't need that fact.

\begin{proof}[Proof of Lemma \ref{lem:mon.pro}] 
Pick $k \in \{1,\ldots,n - 1\}$ and $\eta \in \{0,1\}^{[n]}$ (the case $k=n$ is trivial since both sides of \eqref{RFCondProbability} equal one). Writing the definition of the conditional probability in \eqref{RFCondProbability}, we see that proving the claimed monotonicity is equivalent to proving that the function 
\begin{equation} \label{monotonicityRatio}
\frac{\hP_{n,\gt}(\hat \gd_k = 0, \hat \gd([n] \sm \{k\}) = \eta([n] \sm \{k\}))}{\hP_{n,\gt}(\hat \gd_k = 1, \hat \gd([n] \sm \{k\}) = \eta([n] \sm \{k\}))}
\end{equation}
is non-increasing in $\eta$. Let $a = \max(\{j < k \colon \eta_j = 1\}\cup\{0\}), b = \min(\{j > k \colon \eta_b = 1\}\cup\{n\})$ and $r = \# \{j \in [n] \sm \{k\} \colon \eta_j = 1\}$. Then the ratio in \eqref{monotonicityRatio} equals
\begin{equation}
\frac{e^{\gt r} \hK(b - a)}{e^{\gt (r + 1)} \hK(k - a) \hK(b - k)} = e^{- \gt} \frac{\hK(b - a)}{\hK(k - a)\hK(b - k)}.
\end{equation}
Now, pick $\eta' \in \{0,1\}^{[n]}$ with $\eta \leq \eta'$ and define $a', b'$ for $\eta'$ similarly as for $\eta$ above. Necessarily $0 \leq a \leq a' < k < b' \leq b \leq n$ and we would like to have that 
\begin{equation} \label{eq:logconv}
\frac{\hK(b - a)}{\hK(k - a) \hK(b - k)} \geq \frac{\hK(b' - a')}{\hK(k - a') \hK(b' - k)}.
\end{equation}
This actually follows from the log-convexity of $\hK$. Indeed,
\begin{equation}
\frac{\hK(b - a)}{\hK(k - a) \hK(b - k)} = \frac{\hK((a' - a) + (b - a'))}{\hK((a' - a) + (k - a')) \hK(b - k)},
\end{equation}
which, by log-convexity and since $b - a' \geq k - a'$, is greater than
\begin{equation}
\frac{\hK(b - a')}{\hK(k - a') \hK(b - k)} = \frac{\hK((b - b') + (b' - a'))}{\hK(k - a') \hK((b - b') + (b' - k))},
\end{equation}
which in turn, since $b' - a' \geq b' - k$, is greater than the right-hand side of \eqref{eq:logconv}.
\end{proof}
\begin{lemma}[Comparison between tilting and shifting] \label{lem:comp.tilt.shift}
Suppose that the sequence $\{\hK(n)\}_{n \in \bbN}$ is log-convex. For all $h \in \bbR$, there exist a positive constant $c$ and $\underline \gb,\underline \gt > 0$ such that
\begin{equation}
F(\gb,h;\gt) \geq F(\gb,h + c \gb \gt;0), \qquad 0 \leq \gt \leq \underline \gt, \quad \gb \leq \underline \gb,
\end{equation}
where
\begin{equation} \label{eq:defFtilt}
F(\gb,h;\gt) = \varlimsup_{n \to \infty} \frac{1}{n} \hE_{n,\gt} (\log Z_{n,\gb,h}).
\end{equation}
\end{lemma}
\begin{proof}
Define
\begin{equation} \label{Fnbht}
F_n(\gb,h;\gt) = \frac{1}{n} \hE_{n,\gt} (\log Z_{n,\gb,h}), \qquad n \geq 1.
\end{equation}
The idea is borrowed from~\cite{CadH13} and consists in deriving a differential inequality for $F_n$ involving its partial derivatives with respect to $h$ and $\gt$. As it can be easily checked,
\begin{equation} \label{hDerivative}
\frac{\partial}{\partial h} F_n(\gb,h;\gt) = \frac{1}{n} \hE_{n,\gt} \bE \left\{ \bigg( \sum_{k=1}^n \gd_k \bigg) \frac{1}{Z_{n,\gb,h}} e^{\sum_{k=1}^n (h + \gb \hat \gd_k) \gd_k} \gd_n \right\} = \hE_{n,\gt} \bE_{n,\gb,h} \bigg( \frac{1}{n} \sum_{k=1}^n \gd_k \bigg),
\end{equation}
where $\bP_{n,\gb,h}$ is the Gibbs law as defined in \eqref{eq:defquemeas}, and
\begin{equation} \label{FKGApplication}
\frac{\partial}{\partial \gt} F_n(\gb,h;\gt) = \frac{1}{n} \sum_{k=1}^n \hE_{n,\gt} \{(\hat \gd_k - \hE_{n,\gt}(\hat \gd_k)) \log Z_{n,\gb,h}\}.
\end{equation}
For the random configuration $\hat \gd \in \{0,1\}^{[n]}$, we consider the following functions
\begin{equation}
\hat \gd \mapsto \hat \gd_k - \hE_{n,\gt}(\hat \gd_k) \quad  (\forall k\in[n]) \qquad \text{and}\quad 
\hat \gd\mapsto \log Z_{n, \gb, h}|_{\hat \gd_k=y} \quad(\forall y\geq0)
\end{equation}
Since these functions are non-decreasing in $\hat \gd$ and $\hP_{n,\gt}$ is monotone (this is a consequence of Lemma~\ref{lem:mon.pro}, which we may use since $\hat K$ is log-convex, by \eqref{ReturnTimeDistributions}), we can say by applying the FKG inequality (as stated in~\cite[Theorem~4.11]{GHM1999}) that for all $k\in[n]$ and $y\ge0$,
\beq
\hat \bE_{n, \gt}\{(\hat \gd_k-\hat\bE_{n, \gt}(\hat \gd_k))\log Z_{n, \gb, h}|_{\hat \gd_k=y}\} \geq  \hat \bE_{n, \gt}\{(\hat \gd_k-\hat\bE_{n, \gt}(\hat \gd_k))\} \hat \bE_{n, \gt}\{ \log Z_{n, \gb, h}|_{\hat \gd_k=y}\} = 0.
\eeq
Therefore, we have
\begin{align} \label{eq:deriv_F_theta}
\begin{aligned}
\frac{\partial}{\partial \gt} F_n(\gb,h;\gt) &\geq 
 \frac{1}{n} \sum_{k=1}^n \hE_{n,\gt} \big( (\hat \gd_k - \hE_{n,\gt} (\hat \gd_k)) (\log Z_{n,\gb,h} - \log Z_{n,\gb,h}|_{\hat \gd_k = \hE_{n,\gt} (\hat \gd_k)}) \big)\\
&= \frac{1}{n} \sum_{k=1}^{n-1} \hE_{n,\gt} \Big( (\hat \gd_k - \hE_{n,\gt} (\hat \gd_k)) \int_{\hE_{n,\gt}(\hat \gd_k)}^{\hat \gd_k} \frac{\partial}{\partial y} \big( \log Z_{n,\gb,h}|_{\hat \gd_k = y} \big) \, dy \Big),
\end{aligned}
\end{align}
where we have removed the term $k=n$ which is zero.
For each $k \in [n]$, we introduce the function
\begin{equation} \label{eq:eqfk}
f_k(y) = \frac{1}{\gb} \big( \frac{\partial}{\partial y} \log Z_{n,\gb,h}|_{\hat \gd_k = y} \big) = \bE_{n,\gb,h}|_{\hat \gd_k = y} (\gd_k).
\end{equation}
Notice that $f_k(\hat \gd_k) = \bE_{n,\gb,h} (\gd_k)$. Since the $f_k$'s are non-negative for any $k$ and $\hat \gd_k$ takes values 0 or 1, we obtain
\begin{align} \label{eq:sqLB}
\begin{aligned}
&\frac{\gb}{n} \sum_{k=1}^{n-1} \hE_{n,\gt} \left( (\hat \gd_k - \hE_{n,\gt} (\hat \gd_k))^2 \frac{1}{\hat \gd_k - \hE_{n,\gt} (\hat \gd_k)} \int_{\hE_{n,\gt} (\hat \gd_k)}^{\hat \gd_k} f_k(y) \, dy \right)\\
&\geq C_{n,\gt} \times \frac{\gb}{n} \sum_{k=1}^{n-1} \hE_{n,\gt} \left( \frac{1}{\hat\gd_k - \hE_{n,\gt} (\hat\gd_k)} \int_{\hE_{n,\gt} (\hat\gd_k)}^{\hat\gd_k} f_k(y) \, dy \right),
\end{aligned}
\end{align}
where
\begin{equation}
C_{n,\gt} := \min\{(1 - \hat\bE_{n,\gt} (\hat\gd_k))^2, \hat\bE_{n,\gt} (\hat\gd_k)^2 : k=1, 2, \ldots, { n-1}\}.
\end{equation}
Note that the denominators in \eqref{eq:sqLB} are nonzero and that $C_{n,\gt}$ is positive, because we have previously removed the case $k=n$.

To estimate the integrals in the second line of \eqref{eq:sqLB}, we will show that each $f_k$ is almost constant in $y$. Indeed, let us first write its derivative as
\begin{align}
\begin{aligned}
\frac{\partial}{\partial y} f_k(y) &= \gb \big( \bE_{n,\gb,h}|_{\hat \gd_k = y}(\gd_k^2) - \{\bE_{n,\gb,h}|_{\hat \gd_k = y}(\gd_k)\}^2 \big) = \gb \var_{n,\gb,h}|_{\hat \gd_k = y}(\gd_k)\\
&\leq \gb \bE_{n,\gb,h}|_{\hat \gd_k = y}(\gd_k^2) = \gb f_k(y). 
\end{aligned}
\end{align}
The first line shows that $f_k$ is non-decreasing in $y$ while the second line shows that $e^{-\gb y} f_k(y)$ is non-increasing in $y$. This gives in particular that $f_k(y_1)\ge e^{-\gb} f_k(y_2)$ for all $y_1, y_2\in [0,1]$. Looking back at \eqref{eq:sqLB}, we obtain
\begin{align}\label{eq:sqLowerbound2}
\begin{aligned}
\hE_{n,\gt} \left( \frac{1}{\hat\gd_k - \hE_{n,\gt} (\hat\gd_k)} \int_{\hE_{n,\gt} (\hat\gd_k)}^{\hat\gd_k} f_k(y) \, dy \right) &\geq e^{- \gb} \hat\bE_{n,\gt} (f_k(\hat\gd_k))\\
&\stackrel{\text{ \eqref{eq:eqfk} }}{=} e^{- \gb} \hat\bE_{n,\gt} \bE_{n,\gb,h} (\gd_k).
\end{aligned}
\end{align}
Taking into account \eqref{hDerivative}, \eqref{eq:deriv_F_theta}, \eqref{eq:sqLB} and \eqref{eq:sqLowerbound2} we obtain
\begin{equation}
\frac{\partial}{\partial \gt} F_n(\gb,h;\gt) \geq \gb e^{- \gb} C_{n,\gt} \Big[\frac{\partial}{\partial h} F_n(\gb,h;\gt) { - \frac1n} \Big] \geq c \gb \Big[\frac{\partial}{\partial h} F_n(\gb,h;\gt) { - \frac1n} \Big] \quad \text{ for { $\gb\in (0,1)$}},
\end{equation}
where $c = { (1/e)} \inf_{\gt \le \gt_0} \inf_{n \geq 1} C_{n,\gt}$ is positive for $\gt_0$ small enough (we shall prove this point later) and the term $1/n$ above comes from the fact that the term $k=n$ appears in \eqref{hDerivative} but not in \eqref{eq:sqLB}. 
Therefore we obtain
\beq
\frac{\partial}{\partial \gt} \tilde F_n(\gb,h;\gt) \geq c \gb \frac{\partial}{\partial h} \tilde F_n(\gb,h;\gt), \qquad \mbox{where } \tilde F_n(\gb,h;\gt) =  F_n(\gb,h;\gt) - \frac{h}{n}.
\eeq
Thus the function $g(t) := { \tilde F_n(\gb,h + c \gb \gt (1 - t);\gt t)}$, having non-negative derivative, is non-decreasing in $[0,1]$. Therefore, $g(0) \leq g(1)$, which means that
\begin{equation}
F_n(\gb,h + c \gb \gt;0) \leq F_n(\gb,h;\gt)  + \frac{c\gb\gt}{n}.
\end{equation}
Therefore, we conclude the desired result by taking the superior limit as $n \to \infty$.

We are left with proving that $\inf_{0\le \gt \le \gt_0} \inf_{n\ge 1} C_{n,\gt} >0$ for $\gt_0$ small enough. Since $\ha> 1$ we may apply Lemma \ref{lem:UniformPositive} to the renewal $\htau$ and get that
\beq
\eta := \inf_{0\le \gt \le \gt_0} \inf_{n\ge 1} \hP_\gt(n\in\htau) >0,
\eeq
where $\hP_{\gt}$ is the law of a renewal with inter-arrival distribution $\hP_\gt(\htau_1 = n) = \exp(\gt - F(\gt) n) \hK(n)$, for $n \geq 1$. Also, as we will justify it in \eqref{eq:comp.pgt.pngt}, $\hP_{n,\gt}(k\in\htau) = \hP_\gt(k\in\htau | n\in\htau)$ for $0\le k \le n$. We conclude as follows: for all $0<k<n$ and $0\le \gt \le \gt_0$,
\beq
\hE_{n,\gt}(\hat\gd_k) = \hP_\gt(k\in\htau | n\in\htau) \ge \eta^2,
\eeq
and
\beq
\ba
1 - \hE_{n,\gt}(\hat\gd_k) = \hP_\gt(k\notin\htau | n\in\htau) &\ge \hP_\gt(k-1\in\htau)\hP_\gt(\htau_1 = 2)\hP_\gt(n-k-1\in\htau)\\
&\ge e^{-2\hat F(\gt_0)}K(2)\eta^2.
\ea
\eeq
\end{proof}

\begin{proof}[Proof of Theorem~\ref{thm:smoothing}]
The proof is divided into four steps.

\smallskip

{\noindent \bf Step 1. Lower bound on $F(\gb,h)$.}
Let us cut the system into blocks of size $m \in \bbN$, namely
\begin{equation}
B_j^{(m)} = (jm,(j+1)m], \qquad j \geq 0.
\end{equation}
A block is declared to be {\it good} if on this block the environment is favorable for the polymer. More precisely, pick $a \in (0,1)$ (for the moment its precise value is irrelevant) and define (recall \eqref{eq:def_alt_part_fct}):
\begin{equation} \label{eq:defgoodblock}
\cG = \Big\{ j \geq 0 \colon Z_{jm,(j+1)m,\gb,h} \geq \exp \{a \hE_\gt(\htau_1)^{-1} F(\gb,h;\gt) m\} \Big\}, \qquad \gt > 0,
\end{equation}
where $\hP_{\gt}$ is the law of a renewal with inter-arrival distribution $\hP_\gt(\tau_1 = n) = \exp(\gt - F(\gt) n) \hK(n)$, for $n \geq 1$, and $F(\gb,h;\gt)$ is defined in \eqref{eq:defFtilt}. The factor $\hE_\gt(\htau_1)^{-1}$ in the exponential is essentially harmless and the reason for its presence will appear clearer at the end of the proof. We now consider the blocks for which both endpoints are in $\htau$, i.e.,
\begin{equation}
J = \{ j \geq 0 \colon j m \in \htau, \, (j + 1) m \in \htau\}.
\end{equation}
Then, denote by $\{\gs_k\}_{k \in \bbN}$ the elements of $J \cap \cG$, which form a renewal sequence.

Our first task is to relate $F(\gb,h)$ to the free energy associated to partition functions whose endpoints are in $J \cap \cG$. By Kingman's subadditive ergodic theorem there exists a non-negative number $\cF(\gb,h)$ such that
\begin{equation}
\cF(\gb,h) = \lim_{k\to\infty} \frac{1}{k} \hE(\log Z_{(\gs_k + 1) m}) = \sup_{k\geq1} \frac{1}{k} \hE(\log Z_{(\gs_k + 1) m}),
\end{equation}
and
\begin{equation}
\frac{1}{k} \log Z_{(\gs_k + 1) m} \stackrel{k\to\infty}{\longrightarrow} \cF(\gb,h) \qquad \text{ $\hP$-a.s.\ and in $L^1(\hP)$.}
\end{equation}
The theorem can be applied since the process $\{\hE \log Z_{(\gs_i + 1) m, (\gs_j + 1)m} \}_{i<j \in \bbN}$ is super-additive, while the other conditions are easily verified.
Moreover, by the Renewal Theorem and Proposition~\ref{pr:exist.que.free.energy},
\begin{equation}
\frac{1}{k} \log Z_{(\gs_k + 1) m} = m \frac{\gs_k + 1}{k} \frac{1}{(\gs_k + 1) m} \log Z_{(\gs_k + 1) m} \rightarrow m \hE(\gs_1) F(\gb,h), \quad k \nearrow \infty,
\end{equation}
which yields $\cF(\gb,h) = m \hE(\gs_1) F(\gb,h)$. Finally we obtain
\begin{equation} \label{eq:sm1}
m \hE(\gs_1) F(\gb,h) \geq \hE(\log Z_{(\gs_1 + 1) m,\gb,h}).
\end{equation}

\smallskip

{\noindent \bf Step 2. Lower bound on the right-hand side of \eqref{eq:sm1}.}
We now apply the following strategy: the polymer makes a large jump until $\gs_1 m$ and visits the atypical block $[\gs_1 m, (\gs_1 + 1) m]$. Using \eqref{eq:defgoodblock}, we obtain the following lower bound:
\begin{equation}
\log Z_{(\gs_1 + 1) m,\gb,h} \geq \log K(\gs_1 m) + (h + \gb) + a \hE_\gt(\htau_1)^{-1} F(\gb,h;\gt) m,
\end{equation}
and looking back at \eqref{eq:sm1}, we are left with bounding from below the quantity
\begin{equation}
\hE \log K(\gs_1 m) = \log \hat c_{K} - (1 + \ga) \log m - (1 + \ga) \hE(\log \gs_1).
\end{equation}
By Jensen's inequality, $\hE(\log \gs_1) \leq \log \hE(\gs_1)$. To bound $\hE(\gs_1)$, enumerate $J$ as an increasing sequence $(j_k)_{k\in\bbN}$ and define $R$ as the unique integer for which $\gs_1 = j_R$. By the Markov property, the random variable~$R$ has a geometric distribution with probability of success
\begin{equation} \label{eq.def.pm}
p_m := \hP \Big( Z_{m,\gb,h} \geq \exp\{a \hE_\gt(\htau_1)^{-1} F(\gb,h;\gt) m \} \, \mid \,  m \in \htau \Big).
\end{equation}
The $(j_{k+1} - j_k)_{k \geq 0}$ are i.i.d.~, with $j_0 = 0$. From Wald's equality, $\hE(\gs_1) = \hE(j_1) \hE(R) = \hE(j_1)/p_m$. Then, $\lim_{m\to\infty} \hE(j_1) = \hmu^2$ because $J$ is a renewal process on $\bbN$ with inter-arrival time $j_1$ whose mean value has inverse (by the Renewal Theorem)
\begin{equation}
\lim_{s\to\infty} \hP(s \in J) = \lim_{s\to\infty} \hP(s m \in \htau) \hP(m \in  \htau) = \hmu^{-1} \hP(m \in \htau),
\end{equation}
and the last quantity tends to $\hmu^{-2}$ as $m \to \infty$. Assume for the moment the following entropy estimate: for some constant $c > 0$,
\begin{equation} \label{eq:entropy.estimate}
p_m \geq c \, e^{- h(\hP_{m,\gt} \mid \hP) [1 + o(1)]}, \qquad m \uparrow \infty.
\end{equation}
(for the sake of clarity, we will establish the latter at the end of the proof), from which we obtain
\begin{equation}
\hE(\gs_1) \leq c^{-1} \, \hmu^2 e^{h(\hP_{m,\gt} \mid \hP) [1 + o(1)]}, \qquad m \uparrow \infty.
\end{equation}
Therefore,
\begin{equation} \label{eq:sm2}
\begin{aligned}
\hE (\log Z_{(\gs_1 + 1) m,\gb,h}) &\geq \log c_{\hK} - (1 + \ga) \log(m c^{-1} \hmu^2) + h + \gb \\
&+ a \hE_\gt(\htau_1)^{-1} F(\gb,h;\gt) m - (1 + \ga) h(\hP_{m,\gt} \mid \hP) [1 + o(1)]\\
&= m \Big( a \hE_\gt(\htau_1)^{-1} F(\gb,h;\gt) - (1 + \ga) h(\hP_{m,\gt} \mid \hP)/m + o_{m\uparrow\infty}(1) \Big).
\end{aligned}
\end{equation}
{\noindent \bf Step 3. Conclusion of the proof.}
Using \eqref{eq:sm1} at $h = h_c(\gb)$, \eqref{eq:sm2}, \eqref{eq:def.spec.rel.ent} and letting $m \uparrow \infty$, then $a \uparrow 1$, we get
\begin{equation}
F(\gb,h_c(\gb);\gt) \leq (1 + \ga) \hE_\gt(\htau_1) h_\infty(\gt).
\end{equation}
One can check that \eqref{ReturnTimeDistributions} implies $\hK(n)^2 \leq \hK(n-1) \hK(n+1)$ for $n \geq 2$, so $\{\hK(n)\}_{n\in\bbN}$ is log-convex. Therefore, we may apply Lemma~\ref{lem:comp.tilt.shift}, along with Lemma~\ref{lem:spec.rel.ent.tilt} and the fact that $\hE_\gt(\htau_1) \rightarrow \hmu$ as $\gt \to 0$, to get the desired estimate.

\smallskip

{\noindent \bf Step 4. Proof of \eqref{eq:entropy.estimate}.} Using the fact that for any two probability measures $\mu,\nu$ on the same space and event $E$ with $\nu(E) > 0$ it holds 
\beq
\log\{\mu(E)/\nu(E)\} \geq -(h(\nu \mid \mu) + e^{-1})/\nu(E)
\eeq
\cite[Equation (A.13)]{Gi2007}, it is actually enough to prove that
\begin{equation} \label{eq:cond.entro.est}
\lim_{m\to\infty} \hP_{m,\gt} \Big(Z_{m,\gb,h} \geq \exp\{a \hE_\gt(\htau_1)^{-1} F(\gb,h;\gt)m\} \Big) = 1.
\end{equation}
For the proof of \eqref{eq:cond.entro.est}, the idea is to compute this limit when $\hP_{m,\gt}$ is replaced by $\hP_\gt$ and then relate the two measures. The crucial observation is that for all bounded measurable functions $\phi$,
\begin{equation} \label{eq:comp.pgt.pngt}
\hE_{n,\gt} [\phi(\hat \gd_1,\ldots,\hat \gd_n)] = \frac{1}{\hP_\gt(n \in \htau)} \hE_\gt[\phi(\hat \gd_1,\ldots,\hat \gd_n) \hat \gd_n] = \hE_\gt[\phi(\hat \gd_1,\ldots,\hat \gd_n) \, \mid \, n \in \htau].
\end{equation}
Using the same arguments as in Proposition~\ref{pr:exist.que.free.energy}, we get
\begin{equation} \label{eq:conv.to.olF}
\hP_\gt \text{-a.s.}, \quad \frac{1}{n} \log Z_n \stackrel{n\to\infty}{\longrightarrow} \overline F(\gb,h;\gt) := \varlimsup_{n\to\infty} \frac{1}{n} \hE_\gt(\log Z_n).
\end{equation}
Moreover,
\begin{equation} \label{eq:comp.olF.F}
\overline F(\gb,h;\gt) \geq \hE_\gt(\htau_1)^{-1} F(\gb,h;\gt).
\end{equation}
Indeed, by \eqref{eq:comp.pgt.pngt},
\begin{equation}
\begin{aligned}
\hE_\gt[\log Z_n] &= \hE_\gt[(\log Z_n) \hat \gd_n] + \hE_\gt[(\log Z_n) (1 - \hat \gd_n)]\\
&\geq \hP_\gt(n \in \htau) \hE_{n,\gt}(\log Z_n) - \hE_\gt[(\log Z_n)_-]\\
&\geq \hP_\gt(n \in \htau) \hE_{n,\gt}(\log Z_n) + \log K(n) - h_- \quad \text{ since } Z_n \geq K(n)e^h,
\end{aligned}
\end{equation}
and \eqref{eq:comp.olF.F} follows by considering the superior limit and using the Renewal Theorem. Finally, by using \eqref{eq:comp.pgt.pngt} one more time, we get that
\begin{equation}
\begin{aligned}
&\hP_{m,\gt} \Big(Z_{m,\gb,h} \geq \exp\{a \hE_\gt(\htau_1)^{-1} F(\gb,h;\gt) m\} \Big)\\
&= \frac{1}{\hP_\gt(m \in \htau)} \hP_{\gt} \Big(Z_{m,\gb,h} \geq \exp\{a \hE_\gt(\htau_1)^{-1} F(\gb,h;\gt) m\}, \, m \in \htau \Big),
\end{aligned}
\end{equation}
which goes to $1$ as $m \to \infty$ by virtue of \eqref{eq:conv.to.olF} and \eqref{eq:comp.olF.F} and since $a<1$. This completes the proof.
\end{proof}

\begin{proof}[Proof of Proposition~\ref{LargeBetaRelevance}]
We follow Section~6.2 from~\cite{Gi2011}. We fix $h > 0$ and $\tfrac{1}{1 + \ga} < \gga \leq 1$. Then $\Sigma_\gga := \sum_{n=1}^\infty K(n)^\gga < \infty$. Let $K^{(\gga)}$ be the renewal function defined by $K^{(\gga)}(n) = K(n)^\gga / \Sigma_\gga$ for all $n \in \bbN$ and, for $s < 0$, let $K^{(\gga)}_s$ be the renewal function derived from $K^{(\gga)}$ according to the prescription in \eqref{hTransRenewal}. The $\gga$-th power of the quenched partition function at $h^a_c(\gb) + h$ is bounded as follows:
\begin{equation}
\begin{aligned}
Z_{n,\gb,h_c^a + h}^\gga &\leq \sum_{r=1}^n \sum_{0 = \ell_0 < \ell_1 < \cdots < \ell_r = n} \prod_{i=1}^r K(\ell_i - \ell_{i-1})^\gga e^{\gga (\gb \hat \gd_{\ell_i} + h + h_c^a)}\\
&= \sum_{r=1}^n \sum_{0 = \ell_0 < \ell_1 < \cdots < \ell_r = n} \prod_{i=1}^r K^{(\gga)}(\ell_i - \ell_{i-1}) e^{\gga (\gb \hat \gd_{\ell_i} + h + h_c^a) + \log \Sigma_\gga}\\
& =: \bE_{{K^{(\gga)}_\eta}} \Big( e^{\gga \gb \sum_{k=1}^n \hat \gd_k \gd_k} \Big),
\end{aligned}
\end{equation}
where $\eta = \eta(\gga,h) := \gga (h + h^a_c) + \log \Sigma_\gga$ is negative provided $h$ is close enough to $0$ and $\gb$ is large enough. Consequently, by Lemma~\ref{lem:acc}, $\hat \bE(Z_{n,\gb,h_c^a + h}^\gga)$ does not grow exponentially in $n$ if
\begin{equation} \label{SubexpCondition}
\gb \gga + \log p_\gga(\eta(\gga,h)) \leq 0,
\end{equation}
where $p_\gga$ is the function defined in \eqref{eq:rel_p_I} when the law of the renewal $\tau$ is determined by $K^{(\gga)}_\eta$. In other words,
\begin{equation} 
p_\gga(\eta(\gga,h)) = 1 - I(\gga,h)^{-1}, \text{ with } I(\gga,h) = \hat \bE \bE_{{K^{(\gga)}_\eta}}(|\tilde \tau|) = \sum_{n=0}^\infty \p_{{K^{(\gga)}_\eta}}(n \in \tau) \hP(n \in \htau).
\end{equation}
(Note that here $\htau$ is not the stationary version.) Let $a(\gga,h) := \log p_\gga(\eta(\gga,h))$. Since $\gb + a(1,0) = 0$ (by our characterization of $h_c^a(\gb)$) and $a(\gga,h)$ is continuous in $h$, there exists a $\gga \in (0,1)$ so that relation \eqref{SubexpCondition} holds for some $h > 0$ if
\begin{equation}\label{DerivativeIneq}
\gb + \partial_\gga a(1,0) > 0.
\end{equation}
Using $I(1,0) =\cI(h_a^c)=(1 - e^{- \gb})^{-1}$, we compute 
\begin{equation}
\partial_\gga a(1,0) = (e^\gb + e^{- \gb} - 2) \partial_\gga I(1,0),
\end{equation}
and
\begin{equation}
\partial_\gga I(1,0) = \sum_{n=1}^\infty \hP(n \in \htau) \partial_\gga|_{\gga=1} \bP_{K^{(\gga)}_{\eta(\gga, 0)}}(n \in \tau),
\end{equation}
and
\begin{equation}
\begin{aligned}
&\partial_\gga|_{\gga=1} \p_{{K^{(\gga)}_\eta}} (n \in \tau)\\
&= \sum_{k=1}^n e^{k h_c^a} \sum_{0 = \ell_0 < \ell_1 < \cdots < \ell_k = n} \bigg( k h_c^a + \sum_{i=1}^k \log K(\ell_i - \ell_{i-1}) \bigg) \prod_{i=1}^k K(\ell_i - \ell_{i-1}).
\end{aligned}
\end{equation}
A consequence of Proposition~\ref{pr:ann.crit.curve} and \eqref{eq:cI} is that
\begin{equation}\label{hcAsymptotics}
e^{h_c^a + \gb} = \frac{1}{\bP \times \hP (\tau_1 \in \htau)} + O(e^{- \gb})
\end{equation}
as $\gb \to \infty$, so that in the previous sum the dominant term as $\gb \to \infty$ is the one with $k = 1$, and in fact we can check that
\begin{equation} \label{eq:partialIgamma_aux}
\partial_\gga I(1,0) = e^{h_c^a} \sum_{n=1}^\infty  \hP(n \in \htau) K(n) \{h_c^a + \log K(n)\} + O(\gb e^{- 2 \gb}).
\end{equation}
Indeed, we may write
\begin{equation}
\Big| \partial_\gga I(1,0) -  e^{h_c^a} \sum_{n=1}^\infty \hP(n \in \htau) K(n) \big( h_c^a + \log K(n) \big) \Big| \leq (I) + (II),
\end{equation}
where
\begin{equation}
(I) = \sum_{n=1}^\infty \hP(n \in \htau) \sum_{k\geq2} e^{k h_c^a} k |h_c^a| \sum_{0 = \ell_0 < \ell_1 < \cdots < \ell_k = n} \prod_{i=1}^k K(\ell_i - \ell_{i-1}),
\end{equation}
and
\begin{equation}
(II) = \sum_{n=1}^\infty \hP(n \in \htau) \sum_{k\geq2} e^{k h_c^a} \sum_{0 = \ell_0 < \ell_1 < \cdots < \ell_k = n} \Big( \sum_{i=1}^k - \log K(\ell_i - \ell_{i-1}) \Big) \prod_{i=1}^k K(\ell_i - \ell_{i-1}).
\end{equation}
By interchanging the sums in $k$ and $n$ and bounding $\hP(n \in \htau)$ by one, we get
\begin{equation}
(I) \leq \sum_{k\geq2} e^{k h_c^a} k |h_c^a| = O(\gb e^{- 2 \gb}) \quad \text{ by~\eqref{hcAsymptotics}},
\end{equation}
and
\begin{equation}
(II) \leq (1 - e^{h_c^a})^{-1}\ \bE \Big[ \sum_{1 \leq i \leq N_\gb} (- \log K(T_i)) \ind_{\{N_\gb \geq 2\}} \Big],
\end{equation}
where $N_\gb$ is a geometric random variable with parameter $1 - e^{h_c^a}$, independent from $\tau$, and the $T_i$'s are the increments of $\tau$. Therefore,
\begin{equation}
(II) \leq \text{(cst)}\ \bE(-\log K(T_1)) \times \bE(N_\gb \ind_{\{N_\gb \geq 2\}}) = O(e^{- 2 \gb}), 
\end{equation}
again by~\eqref{hcAsymptotics}.
This settles \eqref{eq:partialIgamma_aux}.

\noindent Since $e^\gb + e^{- \gb} - 2 = e^{\gb} (1 + O(e^{- \gb}))$, the left hand side of \eqref{DerivativeIneq} equals 
\begin{equation} \label{DerivativeIneq2}
\gb + [1+O(e^{-\gb})]\Big[e^{\gb + h_c^a} \Big\{ h_c^a \bP \times \hP(\tau_1 \in \htau) + \sum_{n=1}^\infty \hP(n \in \htau) K(n) \log K(n)\Big\}  + O(\gb e^{-\gb})\Big].
\end{equation}
Now using \eqref{hcAsymptotics} and its consequence, $h_c^a + \gb = - \log \bP \times \hP(\tau_1 \in \htau) + O(e^{- \gb})$, we have that the previous quantity equals
\begin{equation}
- \log \bP \times \hP(\tau_1 \in \htau) + \frac{1}{\bP \times \hP(\tau_1 \in \htau)} \sum_{n=1}^\infty \hP(n \in \htau) K(n) \log K(n) + O(\gb e^{-\gb}),
\end{equation}
and the main claim of the proposition follows.

To prove the assertion about the case of large $\ga$ we will prove that
\begin{align}
&\lim_{\ga\to\infty} \bP \times \hP(\tau_1 \in \htau) = \hK(1), \label{lim1largegb}\\
&\lim_{\ga\to\infty} \sum_{n=1}^\infty \hP(n \in \htau) K_\ga(n) \log K_\ga(n) = 0, \label{lim2largegb}
\end{align}
and note that $\hK(1) \in (0,1)$. For all $\ga > 0$ the quantity $c_\ga := 1/\zeta(1 + \ga)$ is less than 1 and its limit as $\ga \to \infty$ is 1. Also $\lim_{\ga\to\infty} K_\ga(n) = \lim_{\ga\to\infty} c_\ga n^{- 1 - \ga} = \ind_{\{n = 1\}}$ for all $n \in \bbN$. Then in the sum  
\begin{equation}
\bP \times \hP(\tau_1 \in \htau) = \sum_{n=1}^\infty K_\ga(n) \hP(n \in \htau),
\end{equation}
as $\ga \to \infty$, the first term converges to $\hK(1)$, the second to zero, while the rest of the sum converges to zero also as it is bounded by $\int_2^\infty x^{- \ga -1}\,dx < \infty$. Thus \eqref{lim1largegb} follows. For \eqref{lim2largegb} we write
\begin{equation}
\sum_{n=1}^\infty \hP(n \in \htau) K_\ga(n) \log K_\ga(n) = c_\ga \log c_\ga \sum_{n=1}^\infty \hP(n \in \htau) \frac{1}{n^{1 + \ga}} - c_\ga \sum_{n=2}^\infty \hP(n \in \htau) \frac{(1 + \ga) \log n}{n^{1 + \ga}}.
\end{equation}
As $\ga \to \infty$, the first sum goes to $\hK(1)$, while $\lim_{\ga\to\infty} c_\ga \log c_\ga = 0$. In the second sum, the $n = 2$ term converges to zero as $\ga \to \infty$, and the same is true for the rest of the sum because it is bounded by $(1 + \ga) \int_2^\infty x^{- 1 - \ga} \log x\, \dd x < \infty$.
\end{proof}

\subsection{Irrelevance. Proof of Theorem~\ref{thm:irrelevance}}

The proof of Theorem~\ref{thm:irrelevance}, which can be found at the end of this section, essentially relies on Lemma~\ref{lem:second_moment} below, that is a control on the second moment of the partition function at the annealed critical point. Second moment methods and replica arguments have been used in the i.i.d.~case in~\cite{T08}. The extra difficulty in our context is to deal with correlations, which we tackle by means of a decoupling inequality, see \eqref{eq:UBhatU}.
\begin{lemma} \label{lem:second_moment}
Suppose $\ha > 2$ and $\ga < 1/2$. Then, for $\gb$ small enough,
\begin{equation}
\sup_{n\geq1} \hE[(Z_{n,\gb,h_c^a})^2] < \infty.
\end{equation}
\end{lemma}
\begin{proof}[Proof of Lemma~\ref{lem:second_moment}] The proof is split into several steps. In Step~0, we change the partition function to a slightly modified version, which turns out to be more convenient in our context. In Step~1, we provide alternative expressions for the first and second moments of the partition function, using a cluster (or Mayer) expansion. In Step~2, we prove the decoupling inequality in \eqref{eq:UBhatU}, which we use in Step~3 to bound the second moment from above, uniformly in the size of the polymer. In Step~4 we prove that a simplified version of this upper bound is finite. The idea is to use the correlation decay of $\htau$ ($\ha > 2$) to reduce the problem to the usual second moment control for pinning models with i.i.d.~disorder, which in turn makes use of the transient nature of the overlap of two copies of $\tau$ ($0 < \ga < \frac{1}{2}$), see~\cite{Gi2011} and references therein. In Step~5, we conclude the main line of proof. Finally, Step~6 proves a technical point used in Step~4. Note that during this proof we will work only with free partition functions, i.e., $\gd_n$ is removed from the definitions in \eqref{eq:defquemeas}, \eqref{eq:defquepf}, \eqref{eq:defannmeas} and \eqref{eq:defannpf}.\\
\smallskip

\noindent{\bf Step~0.} During this proof we shall use the convention introduced below \eqref{eq:fullpinned}. We introduce, for $h \leq 0$,
\begin{equation}
\bar Z_{n,\gb,h} = \bE_h\Big( e^{\gb \sum_{k=1}^n \gd_k \hat \gd_k}  \Big),
\end{equation}
where $\bP_h$ is defined in \eqref{hTransRenewal}, and note that $\bar Z_{n,\gb,h}(n\in\tau) = Z_{n,\gb,h}(n\in\tau)$, as it was already observed in \eqref{eq:equ.znac}. By decomposing according to the last renewal before $n$, we get
\begin{equation}
\begin{aligned}
Z_{n,\gb,h} &= Z_{n,\gb,h}(n\in\tau) + \sum_{k=1}^n Z_{n-k,\gb,h}(n-k\in\tau) \bP(\tau_1 > k)\\
&= \bar Z_{n,\gb,h}(n\in\tau) + \sum_{k=1}^n \bar Z_{n-k,\gb,h}(n-k\in\tau) \bP(\tau_1 > k)\\
& \leq\ \bar Z_{n,\gb,h}(n\in\tau) + e^{- h} \sum_{k=1}^n \bar Z_{n-k,\gb,h}(n-k\in\tau) \bP_h(\tau_1 > k)\\
&= \bar Z_{n,\gb,h}(n\in\tau) + e^{- h} (\bar Z_{n,\gb,h} - \bar Z_{n,\gb,h}(n\in\tau))\\
&= e^{- h} \bar Z_{n,\gb,h} + (1 - e^{- h}) \bar Z_{n,\gb,h}(n\in\tau) \leq \bar Z_{n,\gb,h}.
\end{aligned}
\end{equation}
Therefore, it is enough to prove that
\begin{equation}
\sup_{n\geq1} \hE(\bar Z_{n,\gb,h_c^a}^2) < \infty.
\end{equation}
\smallskip

\noindent{\bf Step~1.} We first rewrite the first and second moments of the modified partition functions. Namely, for $h < 0$,
\begin{equation} \label{eq:rwt.first.mom}
\hE(\bar Z_{n,\gb,h}) = \bar Z_{n,\gb,h}^a = \sum_{I \subseteq [n]} z^{|I|} U_h(I) \hU(I),
\end{equation}
and
\begin{equation} \label{eq:rwt.sec.mom}
\hE(\bar Z_{n,\gb,h}^2) = \sum_{I,J \subseteq [n]} z^{|I|+|J|} U_h(I) U_h(J) \hU(I \cup J) ,
\end{equation}
where
\begin{equation} \label{eq:def.zbeta}
z = z(\gb) = e^{\gb} - 1, \qquad U_h(I) = \bP_h(I \subseteq \tau), \qquad \hU(I) = \hP(I \subseteq \htau).
\end{equation}
Let us first prove \eqref{eq:rwt.first.mom}. Since $\gd_k \hat \gd_k$ is $\{0,1\}$-valued, we may write
\begin{equation}
\begin{aligned}
\bar Z_{n,\gb,h}^a &= \hE \bE_h(e^{\gb \sum_{1 \leq k \leq n} \gd_k \hat \gd_k}) = \hE \bE_h \left( \prod_{1 \leq k \leq n} (1 + z \gd_k \hat \gd_k) \right)\\
&= \hE \bE_h \Big( \sum_{I \subseteq [n]} z^{|I|} \prod_{i \in I} \gd_i \hat \gd_i \Big),
\end{aligned}
\end{equation}
which gives \eqref{eq:rwt.first.mom} after interchanging sum and expectation. To get \eqref{eq:rwt.sec.mom}, we use the so-called replica trick and obtain
\begin{equation}
\hE(\bar Z_{n,\gb,h}^2)  = \hE \bE_h^{\otimes 2}(e^{\gb \sum_{1 \leq k \leq n} \hat \gd_k (\gd_k + \gd_k')}),
\end{equation}
where $\tau'$ is an independent copy of $\tau$. Then,
\begin{equation}
\begin{aligned}
\hE(\bar Z_{n,\gb,h}^2) &=  \hE \bE_h^{\otimes 2} \left( \prod_{1 \leq k \leq n} (1 + z \gd_k \hat \gd_k) \prod_{1 \leq \ell \leq n} (1 + z \gd_\ell \hat \gd_\ell) \right)\\
&= \hE \bE_h^{\otimes 2} \left( \sum_{I \subseteq [n]} z^{|I|} \prod_{i \in I} \gd_i \hat \gd_i \, \sum_{J \subseteq [n]} z^{|J|} \prod_{j \in J} \gd_j' \hat \gd_j \right)\\
& = \hE \bE_h^{\otimes 2} \Big( \sum_{I,J \subseteq [n]} z^{|I| + |J|} \prod_{i \in I} \gd_i \prod_{j \in J} \gd_j' \prod_{\ell \in I \cup J} \hat \gd_\ell \Big),
\end{aligned}
\end{equation}
and again the result follows by interchanging sum and expectation.\\

\smallskip

\noindent{\bf Step~2.} We now provide an upper bound on $\hU(I \cup J)$, namely
\begin{equation} \label{eq:UBhatU}
\hU(I \cup J) \leq \hU(I) \hU(J) \prod_{1 \leq m \leq g(I,J)} (1 + r(\gD_m)),
\end{equation}
where
\begin{equation}
r(i) = \sup_{j \geq i} \left| 1 - \frac{\hu(i)}{\hu(j)} \right|, \quad \hu(i) = \hP(i \in \htau), \quad i \in \bbN_0,
\end{equation}
$g(I,J)$ is the total number of gaps defined by \eqref{eq:totalgap} below and $(\gD_m)_{1 \leq m \leq g(I,J)}$ is a sequence of {\it gaps} between the sets $I$ and $J$. Informally, we say that there is a gap each time a point in $I$ (resp.\ $J$) is followed by a point in $J$ (resp.\ $I$), see also Figure \ref{fig:gaps}. We give a rigorous definition below.

\smallskip

\noindent {\it Definition of the gaps.} When one of $I, J$ is empty, we let $g(I, J)=0$. Then, the last product in \eqref{eq:UBhatU} is 1 and the inequality is obviously true (note that $\hU(\emptyset) = 1$). When $I$ and $J$ are both non-empty, we give a recursive definition of the gaps using a backward exploration, which will prove to be useful later in the proof of \eqref{eq:UBhatU}. Therefore, we first need to define the last gap of two sets $I$ and $J$ that are both non-empty, which we denote by $\gap(I,J)$. Let us write
\begin{equation}
\begin{aligned}
I &= \{i_1, i_2, \ldots, i_k\} & \text{ with } i_1 < i_2 < \ldots < i_k,\\
J &= \{j_1, j_2, \ldots, j_\ell\} & \text{ with } j_1 < j_2 < \ldots < j_\ell.
\end{aligned}
\end{equation}
If $i_k \neq j_\ell$ then w.l.o.g.~we may assume that $i_k < j_\ell$, in which case we define
\begin{equation}
\gs = \inf\{s \geq 1 \colon j_s > i_k\} \quad (\gs \leq \ell),
\end{equation}
and set
\begin{equation}
\gap(I,J) = j_\gs - i_k, \qquad p(I,J) := j_\gs.
\end{equation}
If $i_k = j_\ell$, then define $\gap(I,J) = 0$ and set $p(I,J) = i_k = j_\ell$. We may now define the sequence of gaps iteratively, as follows: start with $I_0 = I$ and $J_0 = J$ and define for $m \geq 0$,
\begin{equation} \label{eq:defgaps}
\begin{aligned}
&\gD_{m+1} = \gap(I_m, J_m),\\
&I_{m+1} = I_m \setminus [p(I_m,J_m),\infty),\\
&J_{m+1} = J_m \setminus [p(I_m,J_m),\infty),
\end{aligned}
\end{equation}
taking care that the iteration only makes sense until
\begin{equation} \label{eq:totalgap}
g(I,J) = \inf\{m \geq 1 \colon I_{m} = \emptyset \text{ or } J_{m} = \emptyset\},
\end{equation}
which is the total number of gaps.\\
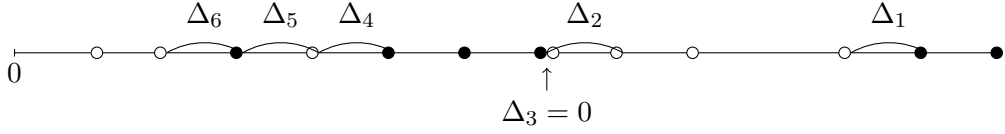
\begin{figure}
\begin{tikzpicture}
\draw (0,0) node[below]{$0$}-- (1,0) ;
\draw (0,-0.05)--(0,0.05);
\draw[o-o] (1,0) -- (2,0) ;
\draw[-*] (2,0) -- (3,0) ;
\draw[-o] (3,0) -- (4,0) ;
\draw[-*] (4,0) -- (5,0) ;
\draw[-*] (5,0) -- (6,0) ;
\draw[-*] (6,0) -- (7,0) ;
\draw[o-o] (7,0) -- (8,0) ;
\draw[-o] (8,0) -- (9,0) ;
\draw[-o] (9,0) -- (11,0) ;
\draw[-*] (11,0) -- (12,0) ;
\draw[-*] (12,0) -- (13,0) ;
\draw (3,0) arc (60:120:1);
\draw (2.5,0.5) node{$\gD_6$};
\draw (4,0) arc (60:120:1);
\draw (3.5,0.5) node{$\gD_5$};
\draw (5,0) arc (60:120:1);
\draw (4.5,0.5) node{$\gD_4$};
\draw (8,0) arc (60:120:1);
\draw (7.5,0.5) node{$\gD_2$};
\draw (12,0) arc (60:120:1);
\draw (11.5,0.5) node{$\gD_1$};
\draw[->] (7,-0.5) node[below]{$\gD_3=0$}--(7,-0.2);
\end{tikzpicture}
\caption{A sequence of gaps, as defined in \eqref{eq:defgaps}. The white and black dots are the points in $I$ and $J$, respectively.}\label{fig:gaps}
\end{figure}

\noindent {\it Proof of the decoupling inequality \eqref{eq:UBhatU}.} Inequality \eqref{eq:UBhatU} follows by iterating the following inequality:
\begin{equation}
\frac{\hU(I \cup J)}{\hU(I) \hU(J)} \leq \frac{\hU(I' \cup J')}{\hU(I') \hU(J')}[1 + r(\gap(I,J))], \qquad \text{ where }
\left\{
\begin{array}{ll}
I' &= I \setminus [p(I,J),\infty),\\
J' &= J \setminus [p(I,J),\infty).
\end{array}
\right.
\end{equation}
Indeed, if $i_k < j_\ell$ (using the notations above), then,
\begin{equation}
\begin{aligned}
\hU(I \cup J) &= \hU(I' \cup J') \hu(j_\sigma - i_k) \prod_{\sigma \leq s < \ell} \hu(j_{s+1}-j_s),\\
\hU(I)\hU(J) &= \hU(I')\hU(J') \hu(j_\sigma - j_{\sigma-1}) \prod_{\sigma \leq s < \ell} \hu(j_{s+1}-j_s).
\end{aligned}
\end{equation}
Therefore,
\begin{equation}
\frac{\hU(I \cup J)}{\hU(I) \hU(J)} \times \left( \frac{\hU(I' \cup J')}{\hU(I') \hU(J')} \right)^{-1} = \frac{\hu(j_\sigma - i_k)}{\hu(j_\sigma - j_{\sigma-1})} \leq 1 + r(j_\sigma - i_k) = 1 + r(\gap(I,J)).
\end{equation}
The case $j_\ell < i_k$ is similar. If $i_k = j_\ell$, we have (assume w.l.o.g.\ that $j_{\ell-1} \leq i_{k-1}$)
\begin{equation}
\hU(I \cup J) = \hU(I' \cup J') \hu(i_k - i_{k-1}), \quad \hU(I) \hU(J) = \hU(I') \hU(J') \hu(i_k - i_{k-1}) \hu(j_\ell - j_{\ell-1}),
\end{equation}
from which we get
\begin{equation}
\frac{\hU(I \cup J)}{\hU(I) \hU(J)} \times \left( \frac{\hU(I' \cup J')}{\hU(I') \hU(J')} \right)^{-1} = \frac{1}{\hu(j_\ell - j_{\ell-1})} \leq 1 + r(0) = 1 + r(\gap(I,J)),
\end{equation}
and \eqref{eq:UBhatU} is proved.\\

\smallskip

\noindent{\bf Step~3.} Recall \eqref{eq:rwt.first.mom}. From Lemma~\ref{lem:bdd.ann.pf.crit} below, we have
\begin{equation} \label{eq:def.s}
s := \sup_{n \geq 1} \bar Z^a_{n,\gb,h_c^a(\gb)} = \sum_{I \subseteq \bbN, |I| < \infty} z^{|I|} U_{h_c^a(\gb)}(I) \hU(I) \leq 2,
\end{equation}
which allows us to define a probability measure on $\{I\in\cP(\bbN)\colon |I|<\infty\}$, which we denote by $ \bbP$. Namely,
\begin{equation} \label{eq:def.doubleP}
\bbP(\{I\}) = s^{-1} z^{|I|} U_{h_c^a(\gb)}(I) \hU(I) \quad \text{ for all $I \subseteq \bbN$ s.t. } |I| < \infty.
\end{equation}

\smallskip

\noindent {\it Upper bound on the second moment.} Using \eqref{eq:rwt.sec.mom}, \eqref{eq:UBhatU}, \eqref{eq:def.s} and \eqref{eq:def.doubleP} we may write
\begin{equation} \label{eq:UB.sec.mom}
\begin{aligned}
\sup_{n \geq 1} \hE(\bar Z_{n,\gb,h_c^a(\gb)}^2) &\leq 4 \bbE^{\otimes 2} \Big( \prod_{1 \leq m \leq g(I,J)} (1 + r(\gD_m)) \Big)\\
&\leq 4 \bbE^{\otimes 2} \left( e^{\sum_{1 \leq m \leq g(I,J)} r(\gD_m)} \right)\\
&= 4 \bbE^{\otimes 2} \left( e^{\sum_{i \geq 0} r(i) \#\{m \geq 1 \colon \gD_m = i\}} \right),
\end{aligned}
\end{equation}
where the gaps $(\gD_m)_{1 \leq m \leq g(I,J)}$ are associated to two sets $I$ and $J$ drawn independently from $\bbP$.\\

\smallskip

\noindent{\bf Step~4.} This is an intermediate step to control the right hand side of \eqref{eq:UB.sec.mom}. If the sets $\{m \geq 1 \colon \gD_m = i\}$ therein were all replaced by $\{m \geq 1 \colon \gD_m = 0\}$, it would be enough to control the following quantity:
\begin{equation} \label{eq:def.cZ}
\cZ := \bbE \left( C^{\#\{m \geq 1 \colon \gD_m = 0\}  } \right) \text{ with }\, C := e^{2 \sum_{i\geq 0} r(i)}.
\end{equation}
We prove in this step that $\cZ$ is finite if $\gb$ is small (recall that $\bbP$ depends on $z = z(\gb)$). Note that $C$ is finite because $\ha > 2$ and $r(i) = O(i^{1 - \ha})$. Indeed, for $i \geq 0$,
\begin{equation}
r(i) = \sup_{j\geq i} \Big| \frac{\hu(j) - \hu(i)}{\hu(j)} \Big| \leq \frac{2}{\inf_{j\geq1} \hu(j)} \times \sup_{j\geq i} \Big| \hu(j) - \frac{1}{\hmu} \Big|,
\end{equation}
which is $O(i^{1 - \ha})$, by \eqref{eq:ren.conv.est}. Then, observe that
\begin{equation}
\#\{m \geq 1 \colon \gD_m = 0\} = |I \cap J|,
\end{equation}
which yields $\cZ = \sup_{n\geq1} \cZ^{(n)}$ with
\begin{equation}
\cZ^{(n)} := s^{- 2} \sum_{I,J\subseteq [n]} z^{|I| + |J|} C^{|I \cap J|} \tilde U_{h_c^a(\gb)}(I) \tilde U_{h_c^a(\gb)}(J),
\end{equation}
where $\tilde U_h(I) = U_h(I) \hU(I)$ is the renewal mass function of $\tilde\tau = \tau \cap \htau$ under the transient renewal process $\tilde \bP_h = \bP_h \times \hP$ ($h < 0$). Using that $|I \cap J| = |I| + |J| - |I \cup J|$, we get
\begin{equation}
\begin{aligned}
\cZ^{(n)} &= \tilde\bE_h^{\otimes 2} \sum_{I,J\subseteq [n]} \Big( \prod_{i\in I} C z \tilde \gd_i \Big) \Big( \prod_{j\in J} C z \tilde \gd_j' \Big) \Big( \frac{1}{C} \Big)^{|I \cup J|}\\
&= \bE_X \tilde \bE_h^{\otimes 2} \sum_{I,J\subseteq [n]} \Big( \prod_{i\in I} C z \tilde \gd_i X_i \Big)  \Big( \prod_{j\in J} C z \tilde \gd_j' X_j \Big),
\end{aligned}
\end{equation}
where $\tilde \tau'$ refers to an independent copy of $\tilde \tau$ and the $X_i$'s are independent Bernoulli random variables with parameter $1/C$ (it is clear from \eqref{eq:def.cZ} that $C > 1$).
Since the $\tilde \gd_i X_i$'s and the $\tilde\gd_j' X_j$'s are $\{0,1\}$-valued, we may write
\begin{equation}
\cZ^{(n)} = \bE_X \tilde \bE_h^{\otimes 2} \Big[ (1 + C z)^{\sum_{1\leq k\leq n} X_k (\tilde \gd_k + \tilde \gd_k')} \Big].
\end{equation}
Integrating over $X$, we get:
\begin{equation}
\cZ^{(n)} = \tilde \bE_h^{\otimes 2} \prod_{1\leq k \leq n} \Big[ 1 + \frac{1}{C} \Big( (1 + C z)^{\tilde \gd_k + \tilde \gd_k'}  - 1 \Big) \Big].
\end{equation}
Using that
\begin{equation}
(1 + C z)^{\tilde \gd_k + \tilde \gd_k'} = (1 + C z \tilde \gd_k)(1 + C z \tilde \gd_k') = 1 + C z (\tilde \gd_k + \tilde \gd_k') + (C z)^2 \tilde \gd_k \tilde \gd_k',
\end{equation}
we obtain that
\begin{equation} \label{eq:UBcZ_aux}
\begin{aligned}
\cZ^{(n)} &= \tilde \bE_h^{\otimes 2} \prod_{1\leq k\leq n} \Big[ 1 + z (\tilde \gd_k + \tilde \gd_k') + C z^2 \tilde \gd_k \tilde \gd_k' \Big]\\
&\leq \tilde \bE_h^{\otimes 2} \prod_{1\leq k\leq n} (1 + z)^{\tilde \gd_k + \tilde \gd_k'} (1 + C z^2)^{\tilde \gd_k \tilde \gd_k'}\\
&= \tilde \bE_h^{\otimes 2} \Big[ e^{\gb \sum_{1\leq k\leq n} (\tilde \gd_k + \tilde \gd_k') + \bar\gb \sum_{1\leq k\leq n} \tilde \gd_k \tilde \gd_k'} \Big],
\end{aligned}
\end{equation}
where (recall \eqref{eq:def.zbeta})
\begin{equation}
\bar\gb = \log(1 + C z^2) \sim C z^2 \sim C \gb^2 \qquad \text{ as } \gb \searrow 0.
\end{equation}
From \eqref{eq:rel_p_I} and Proposition \ref{pr:ann.crit.curve},
\beq
e^\gb \tilde \bP_h (\tilde\tau_1 < \infty) = e^\gb \bP_h\hP (\tilde\tau_1 < \infty)
\left\{
\begin{array}{ll}
<1 & \mbox{if  } h < h_c^a(\gb),\\
= 1 & \mbox{if } h = h_c^a(\gb).
\end{array}
\right.
\eeq
Consequently, the relation
\begin{equation}
\tilde \bP_{\gb,h}(\tilde \tau_1 = n) := e^\gb \tilde \bP_h(\tilde \tau_1 = n) = e^\gb \bP_h \hP( (\tau \cap \htau)_1 = n),\quad h \leq h_c^a(\gb),
\end{equation}
defines a renewal which is recurrent when $h = h_c^a(\gb)$ and transient when $h < h_c^a(\gb)$.
Therefore, we get from \eqref{eq:UBcZ_aux}
\begin{equation} \label{eq:ub.cZn}
\cZ^{(n)} \leq \tilde \bE_{\gb,h}^{\otimes 2} \Big[ e^{\bar\gb \sum_{1\leq k\leq n} \tilde \gd_k \tilde \gd_k' } \Big].
\end{equation}
It is now a standard result about homogeneous pinning models (Theorem 2.7 and the relation (2.13) in~\cite{Gi2011}) that the right hand side of \eqref{eq:ub.cZn} remains bounded provided that
\begin{equation}
e^{\bar\gb} < \frac{1}{\tilde \bP^{\otimes 2}_{\gb,h}((\tilde \tau \cap \tilde \tau')_1 < \infty)}.
\end{equation}
But this is satisfied when $h = h_c^a(\gb)$ and $\gb$ is small enough since, as $\gb \searrow 0$, $\exp(\bar\gb)$ converges to $1$ and
\begin{equation} \label{eq:conv.smallbeta}
\begin{aligned}
\tilde \bP^{\otimes 2}_{\gb,h_c^a(\gb)}((\tilde \tau \cap \tilde \tau')_1 < \infty) &\to \bP^{\otimes 2} \hP^{\otimes 2}((\tau \cap \tau' \cap \htau \cap \htau')_1 < \infty) \quad (\text{ as } \gb \searrow 0)\\
&\leq \bP^{\otimes 2}((\tau \cap \tau')_1 < \infty),
\end{aligned}
\end{equation}
which is strictly less than $1$ because $\ga < 1/2$ (see Proposition~\ref{pr:transient.inter}). Note that the convergence in \eqref{eq:conv.smallbeta} does not seem to follow from simple arguments since Portmanteau's Theorem does not apply and Fatou's lemma would go in the opposite direction. Therefore, we give an argument at the end of the proof, in Step~6, in order not to disrupt the main line of proof.\\

\smallskip

\noindent{\bf Step 5.} We now prove that the right hand side of \eqref{eq:UB.sec.mom} is finite using that $\cZ$ is finite and a stochastic domination argument. For $I, J$ finite subsets of $\bbN$ define  
$$Y_i=\#\{m\geq1 \colon \gD_m = i\}$$
for all $i\in\bbN_0$. We claim that if $I, J$ are independent, each with law $\bbP$, then one can find for all $i \geq 1$ a pair of random variables $Y_i^{(1)}$ and $Y_i^{(2)}$ such that
\begin{equation}\label{eq:stoch.dom}
Y_i \le Y_i^{(1)}+Y_i^{(2)}, \qquad Y_i^{(1)} \preceq 1+ Y_0,\ Y_i^{(2)} \preceq 1+ Y_0.
\end{equation}

We will prove \eqref{eq:stoch.dom} at the end of this step. Apply H\"older's inequality with
\begin{equation}
p_i = \frac{1}{r(i)}\sum_{m\geq0} r(m) \quad \text{ for all } i\in \bbN_0,
\end{equation}
which have $\sum_{i\geq0} \frac{1}{p_i} = 1$, to get
\begin{equation}
\begin{aligned}
\tilde\bE^{\otimes 2} \left( e^{\sum_{i \geq 0} r(i) Y_i} \right) 
\leq &
\prod_{i\geq 0} \left[\tilde\bE^{\otimes 2} \left( e^{Y_i\sum_{m \geq 0} r(m)} \right) \right]^{1/p_i}\\
 \stackrel{\text{by \eqref{eq:stoch.dom}}}{\leq} &\prod_{i\geq 0}  \left[\tilde\bE^{\otimes 2} \left( e^{\{\sum_{m \geq 0} r(m)\}({ Y_i^{(1)}+Y_i^{(2)}})} \right) \right]^{1/p_i}\\
\stackrel{\text{(Cauchy-Schwarz)}}{\leq} &\prod_{i\geq 0}  \left\{ \left[ \tilde\bE^{\otimes 2} \left( e^{2\{\sum_{m \geq 0} r(m)\}Y_i^{(1)}}\right) \right]^{1/2} \left[ \tilde\bE^{\otimes 2} \left( e^{2\{\sum_{m \geq 0} r(m)\}Y_i^{(2)}} \right)\right]^{1/2} \right\}^{1/p_i}\\
\stackrel{\text{by \eqref{eq:stoch.dom}}}{\leq} &\prod_{i\geq 0} \left\{ e^{2\sum_{m\geq 0}r(m)} \tilde\bE^{\otimes 2} \left( e^{2\{\sum_{m \geq 0} r(m)\}Y_0} \right) \right\}^{1/p_i} \\
=&  e^{2\sum_{m\geq 0}r(m)} \tilde\bE^{\otimes 2} \left( e^{2\{\sum_{m \geq 0} r(m)\}Y_0} \right),
\end{aligned}
\end{equation}
which is finite when $\gb$ is small enough, as we have proven in Step 4.\\

We are left with proving \eqref{eq:stoch.dom}. Note that $\bbP$ is the law of a transient renewal with first return time distribution
\begin{equation}
\bbK(n):=\begin{cases} z \bP_{h_c^a(\gb)}(n\in \tau) \hP (n\in \htau) & \text{ if } n\in \bbN,\\ \frac{1}{s} &  \text{ if } n=\infty. 
\end{cases}
\end{equation}
We point out that indeed $s\geq1$, as can be seen by restricting the sum in \eqref{eq:def.s} to the empty set. Then $Y_i\le |I\cap (J-i)|+|(I-i)\cap J|$. If $I\cap (J-i)\ne \emptyset$ and we call $\gz$ its smallest element, then, given $\gz$, the pair $(I-\gz)\cap \bbN, (J-i-\gz)\cap \bbN$ has the same law as $I, J$, by the renewal property. Thus, 
\begin{equation}
|I\cap (J-i)|\preceq 1+|I \cap J|=1+Y_0.
\end{equation}
The same bound applies to $|(I-i)\cap J|$ and we obtain \eqref{eq:stoch.dom} with $Y_i^{(1)}=|I\cap (J-i)|$ and $Y_i^{(2)}=|(I-i)\cap J|$.\\

\noindent{\bf Step 6.} It remains to prove the convergence in \eqref{eq:conv.smallbeta}. Let us use as a shorthand notation:
\begin{equation}
\tilde \bbP_\gb(\tilde \tau_1 = n) := e^\gb \hP \times \bP_{h_c^a}((\tau \cap \htau)_1 = n) = \tilde \bP_{\gb,h_c^a}(\tilde \tau_1 = n), \qquad n \in \bbN.
\end{equation}
We want to prove that
\begin{equation}
\tilde \bbP_\gb^{\otimes 2}((\tilde \tau \cap \tilde \tau')_1 < \infty) \to \tilde \bbP_0^{\otimes 2}((\tilde \tau \cap \tilde \tau')_1 < \infty), \qquad \gb \searrow 0,
\end{equation}
which we do by means of Fourier series (a similar convergence problem was treated in~\cite{P13a} with the same technique). By the Renewal Equation, it is actually enough to prove that
\begin{equation}
\tilde \bbE_\gb^{\otimes 2}(|\tilde\tau \cap \tilde\tau'|) \rightarrow \tilde \bbE_0^{\otimes 2}(|\tilde\tau \cap \tilde\tau'|), \qquad \gb \searrow 0,
\end{equation}
that is, convergence of the series
\begin{equation} \label{L2NormSeries}
\sum_{n\geq0} \tilde \bbP_\gb^{\otimes 2}(n \in \tilde\tau \cap \tilde\tau') = \sum_{n\geq0} \tilde \bbP_\gb(n \in \tilde\tau)^2 \qquad \text{ as } \gb \searrow 0,
\end{equation}
which may be seen as the $L^2$-norm of some function. More precisely, if we define
\begin{equation}
\varphi_\gb(t) = \tilde \bbE_\gb \Big( e^{i t \tilde\tau_1} \Big), \qquad t \in \bbR,
\end{equation}
then
\begin{equation} \label{Inversion_eq}
\tilde \bbP_\gb(n \in \tilde\tau) = \frac{1}{\tilde \bbE_\gb(\tilde\tau_1)} + \frac{1}{2 \pi} \int_{- \pi}^\pi e^{- i n t} 2 \re \Big[ \frac{1}{1 - \varphi_\gb(t)} \Big] \dd t, \qquad \gb \geq 0,
\end{equation}
where $\re(\cdot)$ is the real part of a complex number. The equation~\eqref{Inversion_eq} can be recovered from equation (8) in~\cite[Chap.\ II.9]{Spitzer}. 
Note that the integral in the right-hand side of~\eqref{Inversion_eq} is real and even in $n\in\bbZ$ since then the complex exponential may be replaced by a cosine.
Moreover,
\begin{equation}
\tilde \bbP_\gb(\tilde\tau_1= n) \geq e^{\gb + h_c^a} \bP(\tau_1 = n) \hP(n \in \htau) = \frac{1}{\hmu} e^{\gb + h_c^a} \bP(\tau_1 = n) (1 + o(1)),
\end{equation}
so $\tilde \bbE_\gb(\tilde\tau_1) = \infty$ and the first term in the right hand side of~\eqref{Inversion_eq} is zero. Therefore, the sum in \eqref{L2NormSeries} is related to the $L^2$ norm of $2 \re[(1 - \varphi_\gb(\cdot))^{-1}]$ in $[- \pi,\pi]$ via
\begin{equation}
\| 2 \re[(1 - \varphi_\gb(\cdot))^{-1}] \|_2^2 = \sum_{n\in\bbZ} \tilde \bbP_\gb(|n| \in \tilde\tau)^2 = 1 + 2 \sum_{n\geq1} \tilde\bbP_\gb(n \in \tilde\tau)^2,
\end{equation}
and we only need to show convergence of the $L^2$-norms of $\re[(1 - \varphi_\gb(\cdot))^{-1}]$ to that of $\re[(1 - \varphi_0(\cdot))^{-1}]$ as $\gb \to 0$.

We now observe that
\begin{equation}
\varphi_\gb(t) = \bE \hE \Big( \Psi_\gb e^{i t \tilde\tau_1} \Big), \quad \text{ where } \Psi_\gb := \exp(\gb + h_c^a |\tau \cap (0,\tilde\tau_1]|)
\end{equation}
is such that $\bE\hE(\Psi_\gb) = 1$. Since $h_c^a(\gb) \leq 0$, we get by the Dominated Convergence Theorem that $\lim_{\gb\to0} \varphi_\gb(t) = \varphi_0(t)$. To show the convergence of the $L^2$ norm of $\re[(1 - \varphi_\gb(\cdot))^{-1}]$ to the one of $\re[(1 - \varphi_0(\cdot))^{-1}]$ we show domination by a square integrable function in $[- \pi,\pi]$ for all $\gb \in [0,\overline\gb_0]$ where $\overline\gb_0 > 0$ is fixed. Note that
\begin{equation}
\re[(1 - \varphi_\gb(t))^{-1}] \leq [\re(1 - \varphi_\gb(t))]^{-1}, \qquad t \neq 0,
\end{equation}
and
\begin{equation}
1 - \varphi_\gb(t) = \bE \hE (\Psi_\gb (1 - e^{i t \tilde\tau_1})).
\end{equation}
We treat first values of $t$ close zero. Since $1 - \cos(t \tilde\tau_1) \geq 0$ a.s., we may write for $\gb \in [0,\overline\gb_0]$.
\begin{equation}
\begin{aligned}
\re(1 - \varphi_\gb(t)) &= \bE \hE \Big( \Psi_\gb [1 - \cos(t \tilde\tau_1)] \Big)\\
&\geq \sum_{1\leq n\leq 1/|t|} (1 - \cos(t n)) \bE \hE \Big( \exp\{\gb + h_c^a |\tau \cap (0,n]|\} \ind_{\{\tilde\tau_1 = n\}} \Big)\\
&\geq e^{h_c^a(\gb_0)} \Big(\min_{\ell\geq1} \hP(\ell \in \htau) \Big) \sum_{1\leq n\leq 1/t} (1 - \cos(t n)) \bP(\tau_1 = n)\\
&\sim e^{h_c^a(\gb_0)} \Big( \min_{\ell\geq1} \hP(\ell \in \htau) \Big) (c \, |t|^\ga) \qquad \text{ as } t \to 0.
\end{aligned}
\end{equation}
Note that the minimum in the line above is positive as a consequence of the Renewal Theorem ($\ha > 1$). To go from the second to the third line, we restrict the expectation to the event $\{\tau_1 = n\}$, on which $|\tau \cap (0,n]| = 1$, and use the monotonicity of $h_c^a(\gb)$. The last step is a standard Riemman sum approximation. Thus there is $t_0 \in (0,\pi]$ so that for $|t| \leq t_0$ and $\gb \in [0,\overline\gb_0]$ it holds $\re(1 - \varphi_\gb(t)) \geq c \, |t|^\ga$ for some constant $c > 0$. The function $|t|^{- \ga}$ is in $L^2[- \pi,\pi]$ since $\ga < 1/2$.

Then for $t \in [- \pi,\pi] \sm [- t_0,t_0]$ we have
\begin{equation}
\begin{aligned}
\re(1 - \varphi_\gb(t)) &= \bE \hE \Big( \Psi_\gb [1 - \cos(t \tilde\tau_1)] \Big)\\
&\geq (1 - \cos(t)) e^{h_c^a(\gb_0)} \bP \hP(\tilde\tau_1 = 1) \geq (1 - \cos(t_0)) e^{h_c^a(\gb_0)} \bP \hP(\tilde\tau_1 = 1).
\end{aligned}
\end{equation}
The last quantity does not depend on $t$ and is positive. This completes the proof.
\end{proof}


\begin{proof}[Proof of Theorem~\ref{thm:irrelevance}] The proof can be done by following the same steps as in the proof of Theorem~4.5 in~\cite{Gi2011}. The main ingredient is the fact that the sequence $(X_n)_{n\in\bbN} := (\bar Z_{n,\gb,h_c^a})_{n\in\bbN}$ is uniformly integrable, which is proven in Lemma~\ref{lem:second_moment}. The fact that our disorder is correlated does not necessitate any change here. However, it is useful that the disorder is bounded. The conclusion of Lemma~4.6 in~\cite{Gi2011}, needed in the proof, still holds because $(X_n)_{n\in\bbN}$ satisfies $\hE(X_n) = \bE_{h_c^a} \hE(e^{\gb \sum_{k=1}^n \gd_k \hat \gd_k}) \geq 1,$ and this is enough for the proof of the theorem.
\end{proof}


\begin{lemma} \label{lem:bdd.ann.pf.crit}
The sequence $(\bar Z_{n,\gb,h_c^a})_{n\in\bbN}$ is bounded (by $2$).
\end{lemma}

\begin{proof}[Proof of Lemma~\ref{lem:bdd.ann.pf.crit}] Recall that $\tilde\tau = \tau \cap \htau$, and decompose $Z_{n,\gb,h_c^a}$ according to the events $\{\tilde\tau_1 > n\}$ and $\{\tilde\tau_1 \leq n\}$ to get
\begin{equation}
\begin{aligned}
&\bar Z_{n,\gb,h_c^a} \\
&= \bP_{h_c^a}\hP (\tilde\tau_1 > n) +  \sum_{1\leq k\leq n} \sum_{0=i_0<i_1<\ldots<i_k\leq n} \Big( \prod_{j=1}^k e^\gb \bP_{h_c^a} \hP (\tilde\tau_1 = i_j - i_{j-1}) \Big) \bP_{h_c^a} \hP (\tilde\tau_1 > n-i_k)\\
&\leq \bP_{h_c^a} \hP (\tilde\tau_1 > n) +  \sum_{1\leq k\leq n} \sum_{0=i_0<i_1<\ldots<i_k\leq n} \Big( \prod_{j=1}^k e^\gb \bP_{h_c^a} \hP (\tilde\tau_1 = i_j - i_{j-1}) \Big) e^{\gb} \bP_{h_c^a} \hP (\tilde\tau_1 > n-i_k)\\
&= \bP_{h_c^a} \hP (\tilde\tau_1 > n) + \tilde\bP_{\gb}(\tilde\tau_1 \leq n) \leq 2,
\end{aligned}
\end{equation}
where $\tilde\bP_{\gb}$ has been defined in \eqref{tildeRenewal}.
\end{proof}


\appendix

\section{results on renewals and homogeneous pinning} \label{Facts}
We collect here a few results on renewal processes and the homogeneous pinning model.

\begin{proposition}[See~\cite{A2003} Proposition~2.4, Chapter~I] \label{meanReturns}
If $\tau$ is a (possibly transient) renewal with $\tau_0 = 0$, then $\bE(|\tau|) = \bP(\tau_1 = \infty)^{-1}$.
\end{proposition}


\begin{proposition}[See~\cite{Gi2007} Theorem~A.4, Appendix~A.5.2]
If $\tau$ is a transient renewal that satisfies \eqref{eq:defK_regvar} then
\begin{equation}
\bP(n \in \tau) \sim \frac{\bP(\tau_1 = n)}{\bP(\tau_1 = \infty)^2} \quad \text{as } n \to \infty.
\end{equation}
\end{proposition}


\begin{proposition} \label{RMassFunctionProp}
If $\tau$ is a recurrent renewal that satisfies \eqref{eq:defK_regvar} then as $n \to \infty$,
\begin{equation} \label{RMassFunction}
\bP(n \in \tau) \sim
\begin{cases} 
\frac{K(n)}{(\bar K(n))^2} & \text{ if } \ga = 0,\\
\frac{C_\ga}{L(n) n^{1 - \ga}} & \text{ if } \ga \in (0,1),\\
\left\{ \sum_{j=0}^n \bar K(j) \right\}^{-1} & \text{ if } \ga \geq 1,
\end{cases}
\end{equation}
with $\bar K(j) = \sum_{r=j+1}^\infty K(r)$ for each $j\in\bbN_0$ and $C_\ga = \ga \sin(\pi \ga)/\pi$. 
\end{proposition}


When $\ga = 0$, the sequence $(\bar K(n))_{n\in\bbN}$ is slowly varying (\cite{BGT89}, Proposition~1.5.9 (b)), while when $\ga = 1$, the sequence $(\sum_{j=0}^ n \bar K(j))_{n\in\bbN}$ is slowly varying (\cite{BGT89}, Propositions~1.5.10 and~1.5.9 (a)).

The proof of the statement for $\ga = 0$ is in~\cite[Theorem~1.1]{Na2012}, for $\ga \in (0,1)$ in~\cite[Theorem~1.1]{GL63}, for $\ga = 1$ in~\cite[Theorem~8.7.5]{BGT89}, while for $\ga > 1$ it is the Renewal Theorem.


\begin{proposition} \label{pr:transient.inter}
Let $\tau$ and $\tau'$ be two independent copies of a recurrent renewal process that satisfies \eqref{eq:defK_regvar}.
If $\ga < \frac{1}{2}$, or $\ga = \frac{1}{2}$ and $\sum_{n\geq1} n^{-1} L(n)^{-2} < \infty$, then $\tau \cap \tau'$ is transient.
\end{proposition}

\begin{proof} Use that $\bE^{\otimes 2}(|\tau \cap \tau'|) = \sum_{n\geq0} \bP(n \in \tau)^2$, in combination with Proposition~\ref{RMassFunctionProp}.
\end{proof}


Before stating the next lemma, we recall that an $\ga$-stable random variable $X$ has three parameters $\kappa\in[-1,1]$, $\sigma\ge 0$ and $m\in \bbR$ (skewness, scale and shift, respectively) appearing in its characteristic function:
\beq \label{eq:stable.cf}
\bE(e^{i\theta X}) = 
\left\{
\begin{array}{ll}
\exp(-\sigma |\theta| (1 + i\kappa (2/\pi) \sign(\theta) \log|\theta|) + im\theta) & \text{ if }\ga = 1,\\
\exp(-\sigma^\ga |\theta|^\ga (1 - i\kappa \sign(\theta) \tan(\pi \ga /2)) + im\theta) &  \text{ if } \ga\in(0, 1)\cup(1,2),
\end{array}
\right.
\eeq
see Definition~1.1.6 in \cite{ST94}. If $\kappa =1$ (resp.\ $-1$), $X$ is said to be totally skewed to the right (resp.\ left).

\begin{lemma} \label{lem:conv.power.tauk}
If $r > 0$, then
\begin{equation}
\bE(\tau_k^{- r}) \sim \left\{
\begin{array}{ll}
(\mu k)^{- r} & \quad \mbox{if } \ga > 1\\
\bE(X_\ga^{- r}) k^{- r/\ga} & \quad \mbox{if } \ga \in (0,1)
\end{array}
\right.
, \quad k \to \infty,
\end{equation}
where $X_\ga$ is an $\ga$-stable random variable totally skewed to the right, with scale parameter $\gs > 0$ depending on the distribution of $\tau_1$ and shift parameter $0$.
\end{lemma}

\begin{proof}
If $\ga > 1$, the result follows by bounded convergence since, by the Renewal Theorem, $(\tau_k/k)^{- r}$ converges $\bP$-a.s.~to $\mu^{- r}$ and is bounded from above by $1$. If $\ga \in (0,1)$, we use that $\tau_k/k^{1/\ga}$ converges to an $\ga$-stable random variable $X_\ga$. The only complication is that $(\tau_k/k^{1/\ga})^{- r}$ is not bounded, but the result still holds by uniform integrability, namely, by Exercise~3.2.5 in~\cite{Du10}, it is enough to show that for some $\gga > r$ we have
\begin{equation} \label{eq:UItau}
\sup_{k\geq1} \bE(\{\tau_k/k^{1/\ga}\}^{- \gga}) < \infty.
\end{equation}
To show this, first note that 
\begin{equation} \label{ExpectationTail} 
\bE(\{\tau_k/k^{1/\ga}\}^{- \gga}) = \int_0^\infty \bP(\{\tau_k/k^{1/\ga}\}^{- \gga} > t) \, \dd t = \int_0^\infty \bP(\tau_k < k^{1/\ga} t^{- 1/\gga}) \, \dd t.
\end{equation}
With the use of Chernoff's bound, the probability inside the integral is bounded as
\begin{equation} \label{LDPBound}
\bP(\tau_k/k < k^{(1/\ga) - 1} t^{- 1/\gga}) \leq \exp \Big\{ - k \sup_{\gl\leq0} \{\gl  k^{(1/\ga) - 1} t^{- 1/\gga} - \log M(\gl)\} \Big\},
\end{equation}
where $M(\gl) := \bE(e^{\gl \tau_1})$. A standard Tauberian argument~\cite[(5.22) of Chapter~XIII]{Feller71} shows that there exists a constant $C > 0$ such that $M(\gl) \leq \exp({- C |\gl|^\ga})$ for all $\gl \leq 0$. This implies, for  $x > 0$, the following bound 
\begin{equation}
\sup_{\gl\leq0} \{\gl x - \log M(\gl)\} \geq \sup_{\gl\leq0} \{\gl x + C |\gl|^\ga\} = C_1 x^{- \ga/(1 - \ga)},
\end{equation}
where $C_1 := (1 - \ga) C^{(1 - \ga)^{-1}} \ga^{\ga (1 - \ga)^{-1}} > 0$. Consequently, the probability in \eqref{LDPBound} is bounded from above by $\exp\{- C_1 t^\frac{\ga}{(1 - \ga) \gga}\}$, which completes the proof.
\end{proof}


\begin{lemma} \label{FreeEnergyU}
Let $I$ be a compact subset of $(0,\infty)$, $\{K_\gga \colon \gga \in I\}$ a family of transient renewal inter-arrival laws with $K_\gga(\infty) = 1 - e^{- \gga}$ and mass renewal function $\{u_\gga(n)\}_{n\ge 0} = \{\bP_{K_\gga}(n\in\tau)\}_{n\ge 0}$ for each $\gga \in I$, and such that there are $\ga \geq 0$, $c_1,c_2 > 0$ and a slowly varying function $L$ so that
\begin{equation} \label{ReturnTimeUBound}
c_1 \frac{L(n)}{n^{1 + \ga}} \leq u_\gga(n) \leq c_2 \frac{L(n)}{n^{1 + \ga}}
\end{equation}
for all $n \geq 1$ and $\gga \in I$. Let $F_\gga$ be the free energy corresponding to the homopolymer defined by $K_\gga$. Then there are $C_1,C_2 > 0$ and a slowly varying function $\hat L$ so that 
\begin{equation} \label{FreeEnergyUBound} 
C_1 \leq \frac{F_\gga(\gga + h)}{h^{(1/\ga) \vee 1} \hat L(1/h)} \leq C_2
\end{equation}
for all $h \in (0,1]$ and $\gga \in I$. For $\ga = 0$, \eqref{FreeEnergyUBound}  means that for $h \searrow 0$, $F_\gga(\gga + h)$ vanishes faster than any polynomial.
\end{lemma}

Recall that $F_\gga$ is zero exactly in $(- \infty,\gga]$ and positive elsewhere. 

\begin{proof} For $h > 0$, $F_\gga(\gga + h)$ is the unique solution in $x$ of the equation 
\begin{equation}
\sum_{n=1}^\infty K_\gga(n) e^{- n x} = e^{- (\gga + h)},
\end{equation}
which we write as 
\begin{equation} \label{FreeEnergyEquation}
\Psi_\gga(x) = 1 - e^{- h},
\end{equation}
with
\begin{equation}
\Psi_\gga(x) = 1 - e^\gga \sum_{n=1}^\infty K_\gga(n) e^{- n x}.
\end{equation}
Now for any function $f : \bbN_0 \to [0,\infty)$, we define $\hat f(z) = \sum_{n=0}^\infty f(n) z^n$ for all $z \in [0,1]$. Then the equality 
\begin{equation}
u_\gga(n) = \ind_{\{n = 0\}} + \sum_{j=1}^n K_\gga(j) u_\gga(n-j)
\end{equation}
gives 
\begin{equation}
\hK_\gga(z) = 1 - \frac{1}{\hat u_\gga(z)}.
\end{equation} 
In particular, $e^{- \gga} = 1 - (\hat u_\gga(1))^{-1}$, so that 
\begin{equation}\label{PsiAndA}
\Psi_\gga(x) = 1 - e^\gga \hK_\gga(e^{- x}) = e^\gga \left( \frac{1}{\hat u_\gga(e^{- x})} - \frac{1}{\hat u_\gga(1)} \right) = \frac{e^\gga}{\hat u_\gga(e^{- x}) \hat u_\gga(1)} \, A_\gga(e^{- x}),
\end{equation}
with
\begin{equation}
A_\gga(z) := \hat u_\gga(1) - \hat u_\gga(z) = (1 - z) \sum_{n=1}^\infty u_\gga(n) \sum_{k=0}^{n-1} z^k = (1 - z) \sum_{k=0}^\infty z^k \sum_{n=k+1}^\infty u_\gga(n).
\end{equation}
Of interest to us is the behavior of $\Psi_\gga$ close to 0, and thus of $A_\gga$ close 1. The following claim addresses the issue. To state it, we let $m := \sum_{n=1}^\infty L(n)/n^\ga$.\\

\smallskip

\textsc{Claim}: (a) If $\ga = 0$, then there are $0 < C_3 < C_4$ so that
\begin{equation} \label{claim_a_zero}
C_3 \leq \frac{A_\gga(z)}{L_0((1 - z)^{-1})} \leq C_4,
\end{equation}
for all $z \in [1/2,1]$, where $L_0$ is the slowly varying function defined in \eqref{L0sv}.

\smallskip

(b) If $\ga \in (0,1)$, then there are $0 < C_3 < C_4$ so that 
\begin{equation}
C_3 \leq \frac{A_\gga(z)}{(1 - z)^\ga L((1 - z)^{-1})} \leq C_4,
\end{equation}
for all $z \in [1/2,1]$.

\smallskip

(c) If $\ga = 1$ and $m = \infty$, then there is a slowly varying function $L_1$ and $0 < C_3 < C_4$ so that
\begin{equation}
C_3 \leq \frac{A_\gga(z)}{(1 - z) L_1((1 - z)^{-1})} \leq C_4,
\end{equation}
for all $z \in [1/2,1]$.

\smallskip

(d) If $m < \infty$ then there are $0 < C_3 < C_4$ so that 
\begin{equation}
C_3 \leq \frac{A_\gga(z)}{1 - z} \leq C_4,
\end{equation}
for all $z\in[1/2, 1]$.\\

\medskip

\textsc{Proof of the claim}: By the bounds we have on $u_\gga$, it suffices to examine the behavior of
\begin{equation}
Q(z) := \sum_{k=0}^\infty z^k \sum_{n=k+1}^\infty \frac{L(n)}{n^{1 + \ga}}.
\end{equation}
Denote by $q_k$ the coefficient of $z^k$ in this power series. We have $q_k \sim \frac{L(k)}{\ga k^\ga}$, for $\ga > 0$, by Proposition~1.5.10 in~\cite{BGT89},  while for $\ga = 0$, $q_k$ is slowly varying (Proposition~1.5.9.b in~\cite{BGT89}). Thus,
\begin{equation}
\sum_{k=0}^r q_k \, \,
\begin{cases} 
\sim r \, q_r & \text{ if } \ga = 0,\\
\sim \frac{L(r) r^{1 - \ga}}{\ga (1 - \ga)}, & \text{ if } \ga \in (0,1),\\
\text{is slowly varying } & \text{ if $\ga=1$ and } m = \infty,\\
\to m & \text{ if } m < \infty.
\end{cases}
\end{equation}
This follows from Proposition~1.5.8 and Proposition~1.5.9(a) in~\cite{BGT89}. Then, parts (a)-(c) of the claim follow from Corollary~1.7.3 in~\cite{BGT89}, while for the case $m < \infty$ we just note that $Q(z)$ maps $[1/2,1]$ to a compact set of $(0,\infty)$. The corollary specifies that for $L_0$  in \eqref{claim_a_zero} we can take
\begin{equation} \label{L0sv}
L_0(y) = q_{[y]}
\end{equation}
for all $y \in [0,\infty)$.\\

We continue with the proof of the lemma. The claim above and \eqref{PsiAndA} give that there are constants $0 < C_5 < C_6$ and a slowly varying function $L_2$ so that
\begin{equation}
C_5 \leq \frac{\Psi_\gga(x)}{x^{\ga \wedge 1} L_2(1/x)} \leq C_6,
\end{equation}
for all $x \in (0, \log 2]$ and $\gga \in I$.

Let $C > 0$ be fixed. For $\ga > 0$, by Proposition~1.5.15 in~\cite{BGT89}, there is a slowly varying function $\hat L$ so that a solution $x_C(h)$ of $x^{\ga \wedge 1} L_2(1/x) = (1 - e^{- h})/C$ is asymptotically equivalent to a constant multiple of $h^{1 \vee 1/\ga} \hat L(1/h)$ as $h \to 0^+$ ($\hat L$ is  the same for all $C$). If $\ga = 0$, call $x_C(h)$ the smallest solution of $L_0(1/x) = (1 - e^{- h})/C$. Then $x_C(h) = 1/L_0^{-1}((1 - e^{- h})/C)$ with $L_0^{-1}$ an obviously defined ``inverse'' of $L_0$. It is easy to see, by bounding $q_k$ from below by $\sum_{n=k+1}^\infty L(n)/n^{1 + \gep}$ for any $\gep > 0$, that each $x_C$ goes to zero faster than any power of $h$. Since $\Psi_\gamma(x)$ is increasing in $x$, for each $\gga \in I$ the solution of \eqref{FreeEnergyEquation} is between $x_{C_6}(h)$ and $x_{C_5}(h)$ ($x_{C_6}(h) < x_{C_5}(h)$), which finishes the proof of the lemma.
\end{proof}


\begin{proposition} \label{pr:asympt.hom.free.energy}
Let $F$ be the free energy of the homogeneous pinning model for the renewal $\tau$ with return exponent $\ga > 1$, as in \eqref{ReturnTimeDistributions}. Then, as $h \searrow 0$,
\begin{equation} \label{eq:asymptoticsFh}
\begin{aligned}
F(h) - h F'(0) &\sim \left\{
\begin{array}{ll}
\frac{1}{2 \mu^3} \var(\tau_1) h^2 & \mbox{ if } \ga > 2\\
\frac{c_K}{2 \mu^3} |\log h| h^2 & \mbox{ if } \ga = 2\\
\frac{c(\ga)}{\mu^{\ga + 1}} \, h^{\ga} & \mbox{ if } \ga \in (1,2)
\end{array}
\right. ,\\
F'(h) - F'(0) &\sim \left\{
\begin{array}{ll}
\frac{1}{\mu^3} \var(\tau_1) h & \mbox{ if } \ga > 2\\
\frac{c_K}{\mu^3} |\log h| h & \mbox{ if } \ga = 2\\
\frac{\ga c(\ga)}{\mu^{\ga + 1}} \, h^{\ga - 1} & \mbox{ if } \ga \in(1,2)
\end{array}
\right. ,
\end{aligned}
\end{equation}
where $c(\ga) = c_K \int_0^\infty (e^{- t} - 1 + t) t^{- (1 + \ga)} \, \dd t$.
\end{proposition}

\begin{proof}
Recall that $F'(0) = 1/\mu$, and define $\phi(F) = \bE(e^{- F \tau_1})$ for $F\geq 0$. Using Lemma \ref{seriesAsymptotics}, we get that as $F\to 0$,
\begin{equation}\label{eq:exp_phi}
\phi(F) = 1 - \mu F + [1 + o(1)] \left\{
\begin{array}{ll}
\frac{1}{2} E(\tau_1^2) F^2 & \mbox{ if } \ga > 2,\\
c_K F^2|\log F|/2 & \mbox{ if } \ga =2,\\
 c(\ga) F^\ga   & \mbox{ if } \ga \in (1,2),
\end{array}
\right.
\end{equation}
and
\begin{equation}\label{eq:exp_phi_prime}
\phi'(F) = - \mu + [1 + o(1)] \left\{
\begin{array}{ll}
E(\tau_1^2) F & \mbox{ if } \ga > 2,\\
c_K F |\log F| & \mbox{ if } \ga =2,\\
\ga c(\ga) F^{\ga-1} & \mbox{ if } \ga \in (1,2),
\end{array}
\right.
\end{equation}
with $c(\ga)$ as in the statement of the proposition. We used  the fact that for the functions $A_2, A_4$ of that lemma it holds $A_2(\ga)=\ga A_4(1+\ga)$, which follows by integration by parts. In combination with $\phi(F(h)) = e^{- h} = 1 - h + \frac{1}{2} h^2 [1 + o(1)]$, see~\cite[Section 2.1, Equation (2.2)]{Gi2007}, relation \eqref{eq:exp_phi} yields the first result. To get the second result, differentiate the previous relation to get $F'(h) \phi' (F(h)) = - \exp(- h)$, which we expand around $h = 0$, and use relation \eqref{eq:exp_phi_prime}.
\end{proof}

\begin{lemma}\label{seriesAsymptotics} Let $(r_k)_{k\ge1}$ be a sequence of positive numbers so that $r_k\sim k^{-\gl}$ as $k\to\infty$ for some $\gl>0$. Then

(i)
\beq
\sum_{k=1}^\infty (1-e^{-xk})r _k\sim \begin{cases} x A_1& \text{ if } \gl>2, \\
 x |\log x| & \text{ if } \gl=2, \\
  x^{\gl-1}A_2(\gl) & \text{ if } \gl\in(1, 2) \\
\end{cases}
\eeq
as $x\to 0^+$, where $A_1=\sum_{k=1}^\infty k r_k$ and $A_2(\gl)=\int_0^\infty (1-e^{-t}) t^{-\gl}\, dt$ are positive and finite.

(ii)

\beq
\sum_{k=1}^\infty (e^{-xk}-1+xk) r_k\sim \begin{cases} x^2 A_3& \text{ if } \gl>3, \\
 x^2|\log x|/2 & \text{ if } \gl=3, \\
  x^{\gl-1}A_4(\gl) & \text{ if } \gl\in(2, 3) \\
\end{cases}
\eeq
as $x\to 0^+$, where $A_3=(1/2)\sum_{k=1}^\infty k^2 r_k$ and $A_4(\gl)=\int_0^\infty (e^{-t}-1+t) t^{-\gl}\, dt$  are positive and finite.

(iii)
\beq
\sum_{k=1}^\infty e^{-xk} r_k \sim
\begin{cases}
 |\log x| & \text{ if } \gl = 1,\\
 A_5(\gl) x^{\gl-1}& \text{ if } \gl \in (0,1)
 \end{cases}
\eeq
as $x\to 0^+$, where $A_5(\gl)=\int_0^\infty e^{-t}t^{-\gl}\, dt$ is positive and finite.
\end{lemma}

%
%
%

\begin{proof}
References for $\gl\in (1,2]$ in (i) and $\gl\in(0,1]$ in (iii) are Corollary~8.1.7 and Theorem~1.7.1 in \cite{BGT89}, respectively. If $\gl>2$ in (i), we have
$$\frac{1}{h}\sum_{k=1}^\infty (1-e^{-hk}) r_k=\sum_{k=1}^\infty \frac{1-e^{-hk}}{hk} k r_k$$
and to the last sum we apply the Dominated Convergence Theorem since $(1-e^{-x})/x$ is positive and bounded from above by 1 while $\sum_{k=1}^\infty k r_k<\infty$. Case (ii) follows from similar arguments and is left to the reader.
\end{proof}

In the following two lemmas we denote by $\bP_\gt$ the renewal such that $\bP_\gt(\tau_1 = n) = e^{\gt - F(\gt) n} K(n)$, where $\gt\le 0$. Note that $\bP_0$ coincides with $\bP$. We also define $\phi_\gt(t) = \bE_\gt(e^{i t \tau_1})$, for $t\in\bbR$.

\begin{lemma} \label{lem:UnifEstimatesCFs}
Let $\ga > 1$.  There exist $\gep>0$, $\gt_0>0$ and $c>0$ such that, for $|t| \le \gep$,
\begin{enumerate}
\item $t/c \le |\im(1- \phi_\gt(t))| \le ct$, uniformly in $0 \le \gt \le \gt_0$,
\item $|\im(\phi_0(t) -\phi_\gt(t))| \le \gep_\gt t$, 
\item $0 \le \re(1- \phi_\gt(t)) \le ct^{2\wedge \ga}{(1+ |\log t|\ind_{\{\ga = 2\}})}$, uniformly in $0\le \gt \le \gt_0$,
\item $|\re(\phi_0(t) -\phi_\gt(t))| \le \left\{ \begin{array}{ll}
c \gt^{\ga / 6} t^\ga& \mbox{if } \ga\in(1,2) \mbox{ and } 0 < \gt^{1/2} \le t,\\
{ c \gt^{1/3} t^2 |\log t|} & \mbox{if } \ga=2 \mbox{ and } 0 < \gt^{1/2} \le t|\log t|^{1/2},\\
\gep_\gt t^2& \mbox{if } \ga> 2,
\end{array} \right.$
\end{enumerate}
where $\lim_{\gt\to 0} \gep_\gt = 0$.
\end{lemma}

\begin{proof} Along the proof $c$ will be a constant uniform in $\gt$ and $t$ which may change from line to line.

\smallskip
{\noindent \bf Proof of (1).} 
We have $-\im(1- \phi_\gt(t)) = \bE_\gt (\sin(t\tau_1))$, which is an odd function of $t$, so we restrict the study to $0\le t\le \gep$. In one direction, we have
\beq
\bE_\gt (\sin(t\tau_1)) = \bE(\sin(t\tau_1)e^{\gt - F(\gt)\tau_1}) \le e^{\gt_0}\bE(\tau_1) t,
\eeq
and on the other one,
\beq
\bE_\gt (\sin(t\tau_1)) = \bE_\gt (\sin(t\tau_1)\ind_{\{\tau_1 \le 1/t\}}) + \bE_\gt (\sin(t\tau_1)\ind_{\{\tau_1 > 1/t\}}).
\eeq
For the second term, we have by using that $\frac{\sin x}{x}$ is bounded on $\bbR\sm\{0\}$
\beq
|\bE_\gt (\sin(t\tau_1)\ind_{\{\tau_1 > 1/t\}})| \le ct\bE_\gt(\tau_1 \ind_{\{\tau_1 > 1/t\}}),
\eeq
while for the first term,
\beq
\bE_\gt (\sin(t\tau_1)\ind_{\{\tau_1 \le 1/t\}}) \ge ct \bE_\gt(\tau_1 \ind_{\{\tau_1 \le 1/t\}}) = ct [\bE_\gt(\tau_1) - \bE_\gt(\tau_1 \ind_{\{\tau_1 > 1/t\}})].
\eeq
We get the desired result provided $\gt_0$ and $\gep$ are small enough, since $\bE_\gt(\tau_1) \to \bE(\tau_1)$ and $\bE_\gt(\tau_1 \ind_{\{\tau_1 > 1/t\}}) \le e^{\gt_0} \bE(\tau_1 \ind_{\{\tau_1 > 1/\gep\}}) = o(1)$ as $\gep\to 0$.

\smallskip
{\noindent \bf Proof of (2).} By using that $\frac{\sin x}{x}$ is bounded on $\bbR\sm\{0\}$, we obtain
\beq
| \im(\phi_0(t) -\phi_\gt(t)) | = | \bE(\sin(t\tau_1) (1-e^{\gt - F(\gt)\tau_1})) | \le ct \bE(\tau_1|1-e^{\gt - F(\gt)\tau_1}|),
\eeq
and the expectation following $t$ is $o(1)$ as $\gt\to 0$ by the Dominated Convergence Theorem. 

\smallskip
{\noindent \bf Proof of (3).}
We have
\beq
0 \le \re(1- \phi_\gt(t)) = \bE_\gt[(1-\cos(t\tau_1))] = \bE[(1-\cos(t\tau_1))e^{\gt-F(\gt)\tau_1}] \le e^{\gt_0} \bE[(1-\cos(t\tau_1))],
\eeq
and the result follows by using a standard Tauberian theorem with the fact that $\frac{1-\cos x}{x^2}$ is bounded on $\bbR\sm\{0\}$.

\smallskip
{\noindent \bf Proof of (4). Case $\ga >2$.}
Reusing the same arguments as in the proofs of (2) and (3), we get
\beq
| \re(\phi_0(t) -\phi_\gt(t)) | \le ct^2 \bE(\tau_1^2|1-e^{\gt - F(\gt)\tau_1}|),
\eeq
and we conclude again with the Dominated Convergence Theorem.

\smallskip
{\noindent \bf Proof of (4). Case $\ga \in (1,2)$.} 
We write 
\beq
\ba
\re (\phi_\gt(t) - \phi_0(t)) 
&= \bE[ (1-\cos(t\tau_1)) (1-e^{\gt -F(\gt)\tau_1}) ]\\
& = e^{\gt}\bE[(1-\cos(t\tau_1)) (1-e^{-F(\gt)\tau_1})] -  (e^\gt - 1)\bE[(1-\cos(t\tau_1)) ],
\ea
\eeq
where both terms in the difference are positive. By (3) the second term is less than $c\gt t^\ga$. We now deal with the expectation of $X_{t,\gt} := (1-\cos(t\tau_1)) (1-e^{-F(\gt)\tau_1})$, which we split:
\beq
\bE[X_{t,\gt}] = \bE[X_{t,\gt} \ind_{\{\tau_1 > (b_\gt t)^{-1}\}}]
+ \bE[X_{t,\gt} \ind_{\{\tau_1 \le (b_\gt t)^{-1}\}}],
\eeq
where $b_\gt = o(1)$ shall be specified later. For the first, we use the rough bound 
\beq
\bE[X_{t,\gt} \ind_{\{\tau_1 > (b_\gt t)^{-1}\}}] \le \bP(\tau_1 > (b_\gt t)^{-1}) \le c b_\gt^\ga t^\ga.
\eeq
For the second term, we use that $F(\theta)\sim_0 \gt/\mu$ and that the functions $\frac{1-e^{-x}}{x}$ and $\frac{1-\cos x}{x^2}$ are bounded on $(0,\infty)$ and $\bbR\sm\{0\}$, respectively, in order to get
\beq
\bE[X_{t,\gt} \ind_{\{\tau_1\le (b_\gt t)^{-1}\}}] 
\le c t^2 \gt \sum_{n\le (b_\gt t)^{-1}} n^{2-\ga} 
 \le ct^2\gt (b_\gt t)^{\ga - 3} \le c t^\ga \gt^{1/2}b_\gt^{\ga-3}.
 \eeq
The last inequality is obtained by using our assumption that $\gt^{1/2} \le t$, and the result follows by choosing $b_\gt = \gt^{1/6}$.

\smallskip
{\noindent \bf Proof of (4). Case $\ga =2$.} The proof is the same as in case $\ga \in(1,2)$, except that the splitting is done according to whether $\tau_1$ is smaller or larger than $b_\gt t |\log t|^{1/2}$.
\end{proof}

\begin{lemma} \label{lem:UniformPositive}
Suppose that $\ga > 1$ and $K(n)> 0$ for all $n\ge 1$. Then $\inf_{0 \le \gt \le \gt_0} \inf_{n\ge 1} \bP_\gt(n\in\tau) > 0$ for $\gt_0$ small enough.
\end{lemma}

\begin{proof}
From our assumption on $K$, $\bP(n\in\tau) >0$ for all $n\ge 1$. Morover, $\bP(n\in\tau)$ converges to $1/\mu$ by the Renewal Theorem. Therefore, $\inf_{n\ge 1} \bP(n\in\tau) > 0$ and it is enough to prove that
\beq
\sup_{n\ge 1} |\bP_\gt(n\in\tau) - \bP(n\in\tau)| \rightarrow 0 \quad \text{as } \gt \to 0.
\eeq
By the inversion formula, which can be recovered from equation (8) in~\cite[Chap.\ II.9]{Spitzer},
\beq
\bP_\gt(n\in\tau) = \frac{1}{\bE_\gt (\tau_1)} + \frac{1}{2\pi} \int_{-\pi}^{\pi} e^{-int}\ 2 \re\Big( \frac{1}{1- \phi_\gt(t)}\Big) \dd t,
\eeq
and we get by the triangular inequality
\beq
\sup_{n\ge 1} |\bP_\gt(n\in\tau) - \bP(n\in\tau)| \le 
\Big| \frac{1}{\bE_\gt (\tau_1)} - \frac{1}{\bE (\tau_1)}\Big|
+
\frac{1}{\pi} \int_{-\pi}^{\pi} \Big| \re\Big( \frac{1}{1- \phi_\gt(t)}\Big) - \re\Big( \frac{1}{1- \phi_0(t)}\Big) \Big| \dd t.
\eeq
By Dominated Convergence, $\bE_\gt (\tau_1)$ converges to $\bE (\tau_1)$ as $\gt \to 0$ so the first term is a $o(1)$ and we now focus on the integral, which we split it in two parts: the first one is the integral over a neighborhood of $0$, say $(-\gep,\gep)$, and the second one is the integral outside this neighborhood. Since the only zeros of $1- \phi_0$ are at $2\pi\bbZ$ (by aperiodicity of $\tau_1$), the second part converges to $0$ simply because $\phi_\gt$ converges uniformly to $\phi_0$. We now deal with the first part, which is more delicate. First, let us write
\beq \label{eq:diff_CF_mass}
\frac{1}{1- \phi_\gt(t)} - \frac{1}{1- \phi_0(t)} = \frac{(\phi_\gt(t) - \phi_0(t)) (1- \bar \phi_\gt(t)) (1- \bar \phi_0(t))}{|1- \phi_\gt(t)|^2\ |1- \phi_0(t)|^2}.
\eeq
Pick $\gep$ as in Lemma~\ref{lem:UnifEstimatesCFs}. From the items (1) and (3) of the latter, the denominator in the right-hand side is always of the order of $t^4$. For the rest we distinguish according to the value of $\ga$. Suppose first that $\ga > 2$. Collecting all items in Lemma~\ref{lem:UnifEstimatesCFs}, we see that for all $|t|<\gep$, the real part of \eqref{eq:diff_CF_mass} is bounded in absolute value by a constant times $\gep_\gt = o(1)$ as $\gt \to 0$, which is enough. Suppose now that $\ga \in (1,2)$. By the items (1) to (3) in Lemma~\ref{lem:UnifEstimatesCFs}, we get for all $|t|<\gep$
\beq
|\re[\eqref{eq:diff_CF_mass}]| \le c
[\
| \re(\phi_\gt(t) - \phi_0(t))| t^{-2} 
+\ 
\gep_\gt t^{\ga - 2}\ 
].
\eeq
The second term gives what we want since $t^{\ga-2}$ is now integrable close to $0$. To deal with the second part, we further distinguish between $\gt^{1/2} \le t$ and $t < \gt^{1/2}$, and get
\beq
\int_0^{\theta^{1/2}} | \re(\phi_\gt(t) - \phi_0(t))| t^{-2} \dd t \le c \int_0^{\theta^{1/2}} t^{\ga-2} \dd t= o(1),
\eeq
\beq
\int_{\theta^{1/2}}^\gep | \re(\phi_\gt(t) - \phi_0(t))| t^{-2} \dd t \le c \gt^{\ga/6} \int_0^{\gep} t^{\ga-2} \dd t = o(1),
\eeq
by using the items (3) and (4) of Lemma~\ref{lem:UnifEstimatesCFs}, respectively.
The case $\ga = 2$ is essentially the same, except that in the three lines above $t^{\ga-2}$ is replaced by $|\log t|$ and the condition $t\le \gt^{1/2}$ by $t|\log t|^{1/2}\le \gt^{1/2}$.
This completes the proof.
\end{proof}



\end{document}